\documentclass[a4paper,11pt,reqno]{amsart}
\usepackage{amssymb,amsmath,amsthm,amsfonts,mathtools,lmodern,mathrsfs}

\usepackage[T1]{fontenc}
\usepackage[utf8]{inputenc}
\usepackage[english]{babel}
\usepackage[top=3cm,bottom=2.5cm,left=2.5cm,right=2.5cm]{geometry}
\usepackage{graphicx}
\usepackage{braket}
\usepackage[autostyle]{csquotes}
\usepackage[hidelinks]{hyperref}
\usepackage[capitalise]{cleveref}
\usepackage{enumerate}
\usepackage{comment}

\theoremstyle{plain}
\newtheorem{theorem}{Theorem}[section]
\newtheorem{lemma}[theorem]{Lemma}
\newtheorem{corollary}[theorem]{Corollary}

\newtheorem{proposition}[theorem]{Proposition}
\theoremstyle{definition}
\newtheorem{definition}[theorem]{Definition}
\newtheorem{remark}[theorem]{Remark}
\newtheorem{example}[theorem]{Example}

\newcommand{\re}{\operatorname{Re}}
\newcommand{\im}{\operatorname{Im}}

\newcommand{\R}{\mathbb R}
\newcommand{\N}{\mathbb N}
\renewcommand{\H}{\mathbb H}

\newcommand{\Z}{\mathbb Z}
\newcommand{\E}{\mathbb E}

\newcommand{\C}{\mathbb C}
\newcommand{\g}{\mathfrak{g}}

\newcommand{\vphi}{\varphi}
\newcommand{\spn}{\mathrm{span}}
\newcommand{\vol}{\mathrm{vol}}

\newcommand{\id}{e}

\newcommand{\Lie}{\mathrm{Lie}}

\newcommand{\de}{\mathrm{d}}

\newcommand{\coloneq}{\coloneqq}

\makeatletter
\newcommand{\exterior}[1]{\mathop{\mathpalette\exterior@{#1}}}
\newcommand{\exterior@}[2]{%
	\raisebox{0.7\depth}{%
		\fontsize{\sf@size}{0}%
		\m@th
		$\ifx#1\displaystyle\textstyle\else#1\fi\bigwedge$}%
	^{\mspace{-2mu}#2}%
	\kern-\scriptspace
}
\makeatother

\author[E. Le Donne, L. Nalon, S. Nicolussi Golo, S. Ryoo]{Enrico Le Donne, Luca Nalon, Sebastiano Nicolussi Golo,
Seung-Yeon Ryoo}

\address[Le Donne]{D\'epartement de Math\'ematiques, Ch. du mus\'ee 23, 1700 Fribourg (CH)}
\email[Le Donne]{enrico.ledonne@unifr.ch}

\address[Nalon]{SISSA, Via Bonomea 265, 34136 Trieste (IT)}
\email[Nalon]{lnalon@sissa.it}

\address[Nicolussi Golo]{D\'epartement de Math\'ematiques, Ch. du mus\'ee 23, 1700 Fribourg (CH)}
\email[Nicolussi Golo]{sebastiano.nicolussigolo@unifr.ch}

\address[Ryoo]{California Institute of Technology, 1200 E California Blvd., Pasadena, CA 91125 (USA)}
\email[Ryoo]{sryoo@caltech.edu}

\thanks{E. Le Donne, L. Nalon, and S. Nicolussi Golo were partially supported by the Swiss National Science Foundation 	(grant 200021-204501 `\emph{Regularity of sub-Riemannian geodesics and applications}').}
\thanks{L. Nalon was partially supported by the Swiss National Science Foundation Postdoc. Mobility Fellowship	(project number P500-2235462 `\emph{Lie groups of polynomial growth}').}
\thanks{S. Ryoo was partially supported by an AMS-Simons Travel Grant.}
\date{\today} 

\begin{document}
	
	\title[Asymptotics of step-2 nilpotent Lie groups]{Asymptotics of Riemannian Lie groups \\ with nilpotency step 2}
	
	\begin{abstract} We derive sharp estimates comparing asymptotic Riemannian or sub-Riemannian metrics in 2-step nilpotent Lie groups. For each metric, we construct a Carnot metric whose square remains at a bounded distance from the square of the original metric. 
 In particular, we deduce the analogue of a conjecture by Burago--Margulis: every 2-step nilpotent Riemannian Lie group is at bounded distance from its asymptotic cone.
 As a consequence, we obtain a refined estimate of the error term in the asymptotic expansion of the volume of (sub-)Riemannian metric balls. To achieve this, we develop a novel technique to efficiently perturb rectifiable curves modifying their endpoints in a prescribed vertical direction.
	\end{abstract}
	\maketitle
	\tableofcontents
	
	\section{Introduction}
 The large-scale geometry of metric groups of polynomial growth is closely related to their asymptotic cones; see \cite{MR623534} and also \cite{MR3267520}. 
 For Riemannian nilpotent Lie groups, asymptotic cones are sub-Riemannian Carnot groups; see \cite{donne2024metricliegroupscarnotcaratheodory, MR741395}.
 We stress that these asymptotic spaces are not Riemannian, unless they are commutative as groups.
 Consequently, it is more natural to consider the broad class of sub-Riemannian Lie groups.
 
 When a sub-Riemannian Lie group is simply connected and with nilpotency step $2$, its asymptotic cone can be realized by a left-invariant metric on the same group. Roughly speaking, this metric is defined as the infimum of the sub-Riemannian length over curves within a subclass called $\Delta_\infty$-horizontal curves; see \eqref{def delta infty}. We refer to this metric as the \emph{canonical asymptotic metric} and we will discuss its details in \cref{section_asymptotic_metric}.

Our first main result provides a sharp estimate comparing the sub-Riemannian metric and its canonical asymptotic metric in each $2$-step nilpotent Lie group.

\begin{theorem} \label{main theorem}
Let $G$ be a simply connected, $2$-step nilpotent Lie group equipped with a left-invariant sub-Riemannian metric $d$. Denote by $d_\infty$ the canonical asymptotic metric on $G$. Then, there exists $C > 0$ such that
\begin{equation} \label{main equation}
\lvert d_\infty(p,q)^2 - d(p,q)^2 \rvert \leq C, \quad \text{for every } p,q \in G.
\end{equation}
In particular
\begin{equation} \label{first consequence}
 \lvert d_\infty(p,q) - d(p,q) \rvert = O\left(\frac{1}{d(p,q)}\right) \to 0, \quad \text{as $d(p,q) \to \infty$.}
\end{equation}
\end{theorem}

The interest in obtaining \eqref{first consequence} is in its implications to the volume growth and to the rate of convergence towards the asymptotic cone. 
Namely, \Cref{main theorem} implies that such sub-Riemannian Lie groups are at bounded distance from their asymptotic cones,
in terms of the Gromov-Hausdorff distance; see \cref{prop_rate_convergence}. In fact, the bound \eqref{first consequence} implies that on every $2$-step nilpotent Riemannian Lie group $(G,d)$ there exists a Carnot metric $d'$ on $G$ and $C>0$ such that
\begin{equation} \label{bound metrics intro}
 \lvert d'(p,q) - d(p,q) \rvert \le C \quad \text{for every $p,q \in G$.}
\end{equation}
We refer to \cref{bounded distance cor} for details. 
Even if \eqref{bound metrics intro} has been conjectured for all groups by Burago--Margulis \cite{MR2332168}, analogous results are known to fail in the more general sub-Finsler setting and in nilpotent finitely generated groups. Indeed, in \cite{MR3153949}, one finds an example of a $2$-step nilpotent Finsler Lie group that is not at bounded distance from its asymptotic cone. However, in \cite{MR4380060}, property \eqref{bound metrics intro} is established for nilpotent groups whose asymptotic cones are strongly bracket-generating nilpotent sub-Finsler Lie groups. It is worth noting that the latter condition restricts the nilpotency step to be at most $2$; see \cite[Appendix~A]{MR3739202} for details. Nevertheless, we point out that \eqref{first consequence} does not extend to the Finsler strongly bracket-generating setting. In fact, in \cref{no bounded carnot}, we will exhibit a Finsler metric $d_F$ on the Heisenberg group such that no Carnot metric $d$ satisfies 
\begin{equation} \label{no vanishing}
 \lvert d(p,q) - d_F(p,q) \rvert \to 0, \quad \text{as $d(p,q) \to \infty$.}
\end{equation}
As a consequence, our results highlight a fundamental difference between Riemannian and Finsler metrics in $2$-step nilpotent Lie groups in terms of their asymptotic geometries.

As a second consequence,
\cref{main theorem} refines the asymptotic expansion of the volume growth of metric balls; see \cref{asymptotic_sphere}. Specifically, we will deduce that in every $2$-step nilpotent Riemannian group $(G,d)$, the asymptotic expansion of the Haar measure $\operatorname{vol}$ of the metric ball $B(r)$ of radius $r$ is given by
\begin{equation} \label{asymptotic_expansion}
\operatorname{vol}(B(r)) = C r^Q + O(r^{Q-2}), \quad \text{as $r \to \infty$,} 
\end{equation}
where $Q$ is the Hausdorff dimension of the asymptotic cone of $(G,d)$. Estimate \eqref{asymptotic_expansion} improves previously known results for $2$-step nilpotent Lie groups established in \cite{MR3153949, MR4380060}. We also obtain a rigidity result by showing that the error term in \eqref{asymptotic_expansion} cannot be $o(r^{Q-2})$, unless $(G,d)$ is a Carnot group, and we deduce the sharpness of \cref{main theorem}; see \cref{sharpness}.

The aforementioned rigidity behavior of 2-step sub-Riemannian metrics does not generalize to sub-Finsler metrics. Indeed, in \cref{prop_ellone}, we will exhibit a Finsler metric $d_F$ and a Carnot metric $d$ on the Heisenberg group $\H$ for which 
$$ d_F(p,q) = d(p,q), \quad \text{for every $p,q \in \H$ such that $d(p,q) \ge 1$.}$$

Our proof of \cref{main theorem} builds upon the explicit formula, in $2$-step nilpotent Lie groups, for absolutely continuous curves in terms of their control; see \cref{lemma integral}. Given an absolutely continuous curve, we develop a technique to perturb its endpoint by a prescribed vertical direction. A key feature of this technique is that it provides a bound, independent of the initial curve, on the energy difference between the original and the perturbed curves. We refer to \cref{new_lemma} for details.

We also apply the same technique to obtain a rigidity result for asymptotic sub-Riemannian metrics. Two left-invariant metrics $d$ and $d'$ defined on a Lie group $G$ are \emph{asymptotic} if
$$ \frac{d'(p,q)}{d(p,q)} \to 1, \quad \text{as $d(p,q) \to \infty$.}$$
On the one hand, every sub-Riemannian metric in a nilpotent Lie group is asymptotic to its canonical asymptotic metric by Pansu's work \cite{MR741395}; see \cref{prop_asymptotic}. We point out that Pansu's metric and the canonical asymptotic metric may differ in nilpotent groups of higher step, as explained in \cref{pansu}.
On the other hand, two asymptotic sub-Riemannian metrics may still have distinct canonical asymptotic metrics, as for example, two different Carnot metrics with the same abelianization norm. Nevertheless, we establish the following result, which is meaningful and new also in the Riemannian case.

\begin{theorem} \label{equivalence quasi_isometries intro}
Let $G$ be a simply connected, $2$-step nilpotent Lie group equipped with sub-Riemannian metrics $d$ and $d'$. Then, the metrics $d$ and $d'$ are asymptotic if and only if there exists $C > 0$ such that
\begin{equation*}
\lvert d'(p,q) - d(p,q) \rvert \leq C, \quad \text{for every } p,q \in G.
\end{equation*}
\end{theorem}

\Cref{equivalence quasi_isometries intro}, with its more general form stated in \Cref{equivalence quasi_isometries}, extends results previously known for the Heisenberg group (see \cite[Theorem 1.2]{MR3673666}) to the broader class of $2$-step nilpotent Lie groups. The asymptotic condition is well understood in terms of the abelianization norms induced by the metrics; see \cite[Proposition~3]{MR3153949}. As a consequence, we establish that a Lie group automorphism 
$\vphi:G\to G$
is a rough isometry for a sub-Riemannian 2-step nilpotent Lie group $G$, i.e.,
 there exists $C>0$ such that $$\lvert d(\vphi(p),\vphi(q)) - d(p,q)) \rvert \le C, \quad \text{for every $p,q \in G$},$$
if and only if the abelianization map of $\vphi$ is a Euclidean isometry. 
Using this, we construct examples of Lie group automorphisms that are rough isometries but are not at bounded distance from any isometry; note the difference with finite-dimensional normed vector spaces \cite[Theorem~3]{MR511409}. We discuss these results in \cref{vast class}.


We point out that \cref{main theorem} is informative of the horoboundary of such spaces. Thanks to \eqref{first consequence}, the metrics $d$ and $d_\infty$ have the same horoboundary; see \cite[Section 5]{MR3673666} for a detailed description of this latter fact. We obtain the following consequence.

\begin{corollary}\label{equivalence of horoboundary}
Let $G$ be a simply connected, $2$-step nilpotent Lie group equipped with a sub-Riemannian metric $d$.
Denote by $d_\infty$ the canonical asymptotic metric on $G$.
Then, the horoboundary of $(G,d)$ coincides with the horoboundary of $(G,d_\infty)$.
\end{corollary}

The usefulness of \cref{equivalence of horoboundary} stems from the fact that horoboundaries of Carnot groups are better understood than those of general (sub-)Riemannian Lie groups. Indeed, since the canonical asymptotic metric of $2$-step nilpotent Lie groups coincides with the metric of a Carnot group, its horoboundary can be effectively described using the techniques developed in~\cite{MR4237371}. In particular, see~\cite[Sections 2.6 and 3.3]{MR4237371}, as well as~\cite{fisher2024metricboundarytheorycarnot}.




The paper is organized as follows. In \cref{preliminaries}, we recall fundamental concepts on sub-Riemannian metrics in Lie groups and present the construction of the canonical asymptotic metric. In \cref{main}, we develop our new perturbation technique for curves (\cref{new_lemma}) and use it to establish an improvement of the Ball-Box Theorem; see \cref{powerful_lemma}. The latter plays a central role in proving \cref{main theorem} and \cref{equivalence quasi_isometries intro}. In \cref{consequences}, we examine the implications of these results for the rate of convergence to the asymptotic cone and the volume growth of metric balls. Additionally, we compute the volume of Riemannian balls in the Heisenberg group, for which a refined version of \eqref{asymptotic_expansion} holds.
Finally, we present an example of a Riemannian metric on the Engel group that is not at bounded distance from its canonical asymptotic metric. The latter example suggests that different techniques may be necessary to address similar questions in higher-step nilpotent groups.

 \emph{Acknowledgments.} The first-named author thanks Michael Freedman, Gabriel Pallier, and Pierre Pansu for discussions on earlier versions of these results. All authors are also grateful to Jeremy Tyson for valuable comments.

	\section{Preliminaries} \label{preliminaries}
	
	\subsection
 {Sub-Riemannian metrics on Lie groups}
 We follow standard references for sub-Rie\-man\-nian geometry: \cite{MR3971262,donne2024metricliegroupscarnotcaratheodory,MR1867362}.
	Let $G$ be a connected Lie group, with identity element $\id$ and associated Lie algebra $\g \coloneq T_\id G$. For $p \in G$, we denote by $L_p$ the left-translation by $p$. A subbundle $\Delta$ of the tangent bundle $TG$ is \emph{left-invariant} whenever
	\begin{equation*}
		\Delta_{p \cdot q} = \de L_p(\Delta_q) \quad \text{for every $p,q \in G$.}
	\end{equation*}
	A left-invariant subbundle $\Delta \subseteq TG$ is \emph{bracket-generating} if $\Lie(\Delta_\id)=\g$, where $\Lie(\Delta_\id)$ is the Lie algebra generated by $\Delta_\id$. Hereafter, we refer to a left-invariant, bracket-generating subbundle $\Delta$ as \emph{horizontal distribution}. Fixed some $\Delta \subseteq TG$, we say that an absolutely continuous curve $\gamma \colon [0,1] \to G$ is \emph{horizontal} (or \emph{$\Delta$-horizontal}) if $\dot\gamma(t)\in \Delta_{\gamma(t)}$ for a.e. $t\in[0,1]$.
	
	Consider a horizontal distribution $\Delta \subseteq TG$ and fix an inner product $\rho_\id$ on $\Delta_\id \subseteq \g$. The inner product $\rho_\id$ extends by left-translation to a metric tensor $\rho$ on $\Delta$, i.e.,
	\begin{equation*}
		\rho_p(X,Y) \coloneq \rho_\id(\de L_p^{-1}(X),\de L_p^{-1}(Y)) \quad \text{for every $p \in G$ and $X,Y \in \Delta_p$.}
	\end{equation*}
	We say that $(\Delta,\rho)$ is a \emph{sub-Riemannian structure} on $G$. 
	
	The choice of a sub-Riemannian structure endows $G$ with the structure of metric space, with \emph{sub-Riemannian metric} defined by setting
	\begin{equation} \label{cc distance}
		d(p,q) = \inf \Set{\ell(\gamma) | \text{$\gamma \colon [0,1] \to G$ is a $\Delta$-horizontal curve, $p,q \in \gamma([0,1])$}},
	\end{equation}
	where $\ell(\gamma)\coloneq\int_0^1 \sqrt{\rho(\dot{\gamma}(t),\dot{\gamma}(t))} \, \de t$ denotes the \emph{length} of the horizontal curve $\gamma$.
	The bracket-generating condition ensures that \eqref{cc distance} indeed defines a distance that induces the manifold topology; in particular, it is finite. Since $\ell(L_h \circ \gamma)=\ell(\gamma)$, the distance $d$ is \emph{left-invariant}:
	\begin{equation*}
		d(h \cdot p,h \cdot q)=d(p,q) \quad \text{for every $h,p,q \in G$.}
	\end{equation*} 	
	Moreover, $G$ is complete and the distance $d$ is geodesic, meaning that the infimum in \eqref{cc distance} is always attained; see \cite[Theorem~7.1.7]{donne2024metricliegroupscarnotcaratheodory}. We say that a $\Delta$-horizontal curve is \emph{length-minimizing} (or \emph{$d$-minimizing}) if it realizes the infimum in \eqref{cc distance}.
	We refer to this class of metric spaces as \emph{sub-Riemannian Lie groups}. The special case $\Delta=TG$ corresponds to the well-known class of \emph{Riemannian Lie groups}. 

	\subsection{Nilpotent Lie groups and asymptotic metrics} \label{section_asymptotic_metric}
 Nilpotent Lie groups, equipped with a sub-Riemannian structure, admit a unique asymptotic cone in the Gromov-Hausdorff sense; see \cite{MR741395}. The resulting asymptotic cone is again a sub-Riemannian Lie group. If $G$ is a simply connected, $2$-step nilpotent Lie group, the asymptotic cone can be realized by a sub-Riemannian structure on $G$ itself, following a canonical construction that we now present.
	
\begin{definition}[Canonical asymptotic metric]
Let $G$ be a nilpotent simply connected Lie group with associated Lie algebra $\g$. Equip $G$ with a sub-Riemannian structure $(\Delta,\rho)$ and denote by $d$ the corresponding sub-Riemannian metric.
 We define 
	\begin{equation} \label{def V}
		V \coloneq V_{(\Delta,\rho)} \coloneq \Set{v \in \Delta_\id | \text{$\rho(v,w)=0$ for every $w \in \Delta_\id \cap [\g,\g]$}},
	\end{equation}
	\begin{equation} \label{def delta infty}
		\Delta_\infty \coloneq \bigcup_{p \in G} \de L_p(V) \subseteq \Delta,
	\end{equation}
 and, for each $p,q\in G$, we define 
	\begin{equation} \label{asymptotic metric}
		d_\infty(p,q) \coloneq \inf \Set{\ell(\gamma) | \text{$\gamma \colon [0,1] \to G$ is a $\Delta_\infty$-horizontal curve, $p,q \in \gamma([0,1])$}} .
	\end{equation}
We refer to $d_\infty$ as the \emph{canonical asymptotic metric} of the sub-Riemannian Lie group $(G,d)$.
\end{definition}
 \begin{proposition} \label{prop_asymptotic}
 Let $G$ be a nilpotent simply connected Lie group, equipped with a sub-Riemannian structure $(\Delta,\rho)$. Denote by $d$ the sub-Riemannian metric and by $d_\infty$ its canonical asymptotic metric. Then, the function $d_\infty$ is the left-invariant sub-Riemannian metric on $G$ associated with the sub-Riemannian structure $(\Delta_\infty,\rho|_{\Delta_\infty})$, the metric
 $d_\infty$ is asymptotic to $d$, and
 \begin{equation} \label{obvious bound}
 d(p,q) \le d_\infty(p,q), \quad \text{for every $p,q \in G$.}
 \end{equation}
 \end{proposition}
\begin{proof}
 By an easy fact (see \cite[Exercise~10.5.7]{donne2024metricliegroupscarnotcaratheodory}), since $\Delta$ is bracket-generating, we have
 \begin{equation*}
 \g= \Delta_\id + [\g,\g] .
 \end{equation*}
	In particular, we have obtained that 
	\begin{equation} \label{v generates}
	\g = V \oplus [\g,\g].
	\end{equation} Consequently, the subspace $V$ and 
 the left-invariant subbundle $\Delta_\infty$
	are bracket-generating; see \cite[Exercise~10.5.8]{donne2024metricliegroupscarnotcaratheodory}. Thus, the function $d_\infty$ is the sub-Riemannian metric associated with $(\Delta_\infty,\rho|_{\Delta_\infty})$. 

 We consider the projection $\pi \colon \g \to V$ parallel to $[\g,\g]$. We observe that $\rho$ is an inner product on $\Delta_\id$ and $V$ is orthogonal to $[\g,\g]$. Therefore, 
	$$ \pi \left( \Set{v \in \Delta_e | \rho(v,v) \le 1} \right) = \Set{v \in V | \rho(v,v) \le 1}. $$ 
 By the general result \cite[Proposition~3]{MR3153949}, the distances are asymptotic. 

 Since $\Delta_\infty \subseteq \Delta$, every $\Delta_\infty$-horizontal curve is also $\Delta$-horizontal. Therefore, we conclude \eqref{obvious bound}.
\end{proof}

The construction of the canonical asymptotic metric is in the spirit of Pansu's description of the asymptotic cone in \cite{MR741395}. In general, to construct the Pansu metric, one takes any $V$ transverse to $[\g,\g]$, together with an adapted linear grading; see \cite[Section~12.4]{donne2024metricliegroupscarnotcaratheodory}.

 \begin{proposition}\label{pansu}
 Given a sub-Riemannian, nilpotent, simply connected Lie group $(G,d)$, its canonical asymptotic space $(G,d_\infty)$ is isometric to the asymptotic cone of $(G,d)$ if and only if the subspace $V$, defined as in \eqref{def V}, is the first layer of a stratification of the Lie algebra of $G$. 
 \end{proposition}
\begin{proof}
On the one hand, assume that the subspace $V$ is the first layer of a stratification. In this case, the pair $(G, d_\infty)$ is a Carnot group. Moreover, the metric spaces $(G, d_\infty)$ and $(G,d)$ have the same asymptotic cones thanks to \cref{prop_asymptotic}. Being a Carnot group, the only asymptotic cone of $(G, d_\infty)$ is itself.

On the other hand, if there is an isometry between $(G,d_\infty)$ and the asymptotic cone $(\hat G,\hat d)$ of $(G,d)$, then by \cite{MR3646026} the sub-Riemannian Lie groups $(G,d_\infty)$ and $(\hat G, \hat d)$ are isomorphic via an isometry. In particular, since $(\hat G, \hat d)$ is a Carnot group by \cite{MR741395}, the subspace $V$ must be the first layer of a stratification.
\end{proof}
 \begin{remark} We stress that, if the nilpotency step is two, each subspace $V$ in direct sum with the derived subalgebra is the first layer of a stratification. In this case, the construction of the canonical asymptotic metric is a particular case of Pansu's construction of the asymptotic cone in \cite{MR741395}.

 Nevertheless, examples of nilpotent Lie groups on which there is a subspace $V$ in direct sum with the derived subalgebra that is not the first layer of a stratification may be found in \cite{MR3793294, MR4490195} and \cite[Exercise~9.5.28]{donne2024metricliegroupscarnotcaratheodory}. Therefore, according to \cref{pansu}, when the nilpotency step is higher than two, the canonical asymptotic space $(G,d_\infty)$ may not be isometric to the asymptotic cone of $(G,d)$. 
 \end{remark}

\subsection{Models for sub-Riemannian 2-step nilpotent Lie groups} \label{model}
Every simply connected, $2$-step nilpotent Lie group with associated Lie algebra $\g$ is isomorphic to 
\begin{equation} \label{def G}
 G:= (\g,*), \quad \text{where} \quad x*y \coloneq x+y+\tfrac{1}{2}[x,y], \qquad \text{ for all } x,y \in \g.
\end{equation}
Moreover, the Lie algebra associated with $(\g,*)$ is $\g$ and the exponential map $\exp \colon \g \to (\g,*)$ is the identity map; see \cite[Theorem~11.2.6]{hilgert2011structure}.

Next, fix a sub-Riemannian structure $(\Delta,\rho)$ on $G$ as in \eqref{def G}. We recall the standard correspondence between horizontal curves and their controls; see \cite[Section~7.2.1]{donne2024metricliegroupscarnotcaratheodory} for details. For every $u \in L^1([0,1],\g)$, we denote by $\gamma_u$ the absolutely continuous curve $\gamma \colon [0,1] \to G$ such that 
\begin{equation} \label{system}
 \begin{cases}
 \gamma(0)=\id, \\ \dot{\gamma}(t)= \de L_{\gamma(t)}(u(t)), \quad \text{for a.e. $t \in [0,1]$.}
 \end{cases}
\end{equation}
Vice versa, for every absolutely continuous curve $\gamma \colon [0,1] \to G$ such that $\gamma(0)=\id$, there exists a unique $u \in L^1([0,1],\g)$ such that \eqref{system} holds. We call $u$ the {\em control} of $\gamma$. The curve $\gamma_u$ is $\Delta$-horizontal if and only if $u(t) \in \Delta$ for a.e. $t \in [0,1]$.

 As a simple consequence of the group law \eqref{def G} of $G$, we get the following formula for the curve $\gamma_u$.

\begin{remark} \label{lemma integral}
	Let $G$ be as in \eqref{def G}, and let $\gamma_u \colon [0,1] \to G$ be an absolutely continuous curve such that \eqref{system} holds. Then
	\begin{equation} \label{formula integral}
		\gamma_u(1) = \int^1_0 u(t) \, \de t + \frac{1}{2} \int^1_0 \left[\int^t_0 u(s) \, \de s, u(t) \right] \, \de t. 
	\end{equation}
\end{remark}

We include here two results that we will need later in this section.

\begin{lemma} \label{lemma_wedge}
Fix $G$ as in \eqref{def G}.
	Let $u,v \in L^1([0,1],\g)$. Assume that 
	\begin{equation*} 
			\int^1_0 v(t) \, \de t = 0 \quad \text{and} \quad \int^1_0 \left[\int^t_0 u(s) \, \de s , v(t)\right] \de t = 0.
	\end{equation*}
	Then $\gamma_{u+v}(1)=\gamma_u(1) + \gamma_v(1)$.
\end{lemma}

\begin{proof}
Integrating by parts, we get
\begin{equation} \label{parts}
 \int^1_0 \left[\int^t_0 v(s) \, \de s , u(t)\right] \de t = - \int^1_0 \left[ v(t), \int^t_0 u(s) \de s \right] \de t = \int^1_0 \left[\int^t_0 u(s) \, \de s , v(t)\right] \de t,
\end{equation}
where we used the fact that $\int^1_0 v(t) \, \de t = 0$. We then compute
\begin{eqnarray*}
 \gamma_{u+v}(1) &\stackrel{\eqref{formula integral}}{=}& \int^1_0 u(t) +v(t) \, \de t + \frac{1}{2} \int^1_0 \left[\int^t_0 u(s) + v(s) \, \de s, u(t)+v(t) \right] \, \de t. \\
 &\stackrel{\eqref{formula integral},\eqref{parts}}{=}& \gamma_u(1) + \gamma_v(1) + \int^1_0 \left[\int^t_0 u(s) \, \de s , v(t)\right] \de t \\
 &=& \gamma_u(1) + \gamma_v(1),
\end{eqnarray*}
 where in the last equality we used the hypothesis $\int^1_0 \left[\int^t_0 u(s) \, \de s , v(t)\right] \de t = 0$.
\end{proof}

\begin{lemma} \label{equivalence_norms}
 Let $\g$ be a $2$-step nilpotent Lie algebra of the form $\g=V \oplus [\g,\g]$ for some $V \subseteq \g$. Let $\rho$ be an inner product on $V$ and $\lvert \cdot \rvert$ be a norm on $[\g,\g]$. Set $m \coloneq \dim([\g,\g])$. There exists $K > 0$ such that, for every $\zeta \in [\g,\g]$, there exist $x_1,\dots,x_{m} \in V$, $y_1,\dots,y_{m} \in V$, and $\alpha_1,\dots,\alpha_{m} \ge 0$ such that
 \begin{enumerate}[(i)]
 \item $\rho(x_k,x_k)=\rho(y_k,y_k)=1$ for every $1 \le k \le m$,
 \item $\rho(x_k,y_k)=0$ for every $1 \le k \le m$,
 \item $\zeta = \sum_{k=1}^{m} \alpha_k[x_k,y_k]$,
 \item $\sum_{k=1}^{m} \alpha_k \le K \lvert \zeta \rvert$.
 \end{enumerate}
\end{lemma}
\begin{proof}
 Let $\left([x_k,y_k]\right)_{k=1}^{m}$, with $x_k,y_k \in V$ for every $1 \le k \le m$, be a basis of $[\g,\g]$ made of simple Lie brackets, which exists because $[\g,\g]=[V,V]$. Up to applying a Gram–Schmidt process and a renormalization, we can assume that $\rho(x_k,y_k)=0$ and $\rho(x_k,x_k)=\rho(y_k,y_k)=1$, for every $1 \le k \le m$. The $\ell_1$ norm $\lVert \cdot \rVert$ with respect to $\left([x_k,y_k]\right)_{k=1}^{m}$ and $\lvert \cdot \rvert$ are equivalent norms on the finite-dimensional vector space $[\g,\g]$. Let $K > 0$ be such that $\lVert \cdot \rVert \le K \lvert \cdot \rvert$. Since $\left([x_k,y_k]\right)_{k=1}^{m}$ is a basis of $[\g,\g]$, every $\zeta \in [\g,\g]$ is uniquely written as $\zeta = \sum_{k=1}^{m} \alpha_k[x_k,y_k]$. Up to switching $x_k$ with $y_k$, we can assume that $\alpha_k \ge 0$ for every $1 \le k \le m$. With this assumption $\sum_{k=1}^{m} \alpha_k = \lVert \zeta \rVert$. Therefore, also \emph{(iv)} is verified.
\end{proof}

\subsection{Energy of curves} In our discussion, it is convenient to consider energy-minimizing curves in sub-Riemannian Lie groups. We summarize some standard terminology and basic facts, and we refer to \cite[Section 7.3]{donne2024metricliegroupscarnotcaratheodory} for details.

\begin{definition} \label{energy}
Let $(\Delta,\rho)$ be a sub-Riemannian structure on a Lie group $G$. Let $\gamma \colon [0,1] \to G$ be a $\Delta$-horizontal curve. We say that $\gamma$ is \emph{energy-minimizing} if 
	\begin{equation} \label{def_minimizing}
	\int^1_0 \rho(\dot{\gamma}(t),\dot{\gamma}(t)) \, \de t= d(\gamma(0),\gamma(1))^2.
 \end{equation} 
 
\end{definition}

A $\Delta$-horizontal curve is energy-minimizing if and only if it is length-minimizing and it is parameterized with constant speed; see \cite[Proposition~3.1.11]{donne2024metricliegroupscarnotcaratheodory}. 
Moreover, we get the following formula for the square of the metric in every sub-Riemannian Lie group $(G,d)$:
\begin{equation} \label{distance_square}
 d(e,p)^2 = \inf \Set{ \int^1_0 \rho_\id(u(t),u(t)) \, \de t \, | \, \text{$u \in L^2([0,1],\Delta_\id)$ such that $\gamma_u(1)=p$}}.
\end{equation}

\subsection{Complexification of real Lie algebras} The key step in the proof of the improved Ball-Box Theorem (\cref{powerful_lemma}) will be to perturb a $\Delta$-horizontal curve in order to vary its end-point by a prescribed vertical direction. In our argument, it is convenient to write these perturbations as Fourier series in the complexified Lie algebra, which we briefly recall. We refer to \cite[Section~5.1.4]{hilgert2011structure} for details.

If $\g$ is a real Lie algebra, we define its \emph{complexification} $\g_\C$ as the complex Lie algebra $$\g_\C \coloneq \Set{ x + iy \, | \, x,y \in \g},$$ where the Lie bracket is defined by setting
\begin{equation*} 
 [x_1 + iy_1,x_2+iy_2] \coloneq [x_1,x_2]-[y_1,y_2] + i\left([x_1,y_2]+[x_2,y_1]\right),
\end{equation*}
for every $x_1,x_2,y_1,y_2 \in \g$. If $z = x + iy \in \g_\C$ for some $x,y \in \g$, we define
\begin{equation*}
 \bar{z} \coloneq x-iy, \qquad \re(z) \coloneq x, \qquad \im(z)\coloneq y.
\end{equation*}
An inner product $\rho$, defined on a subspace $\Delta \subseteq \g$, extends to an Hermitian product on $\Delta_\C$ by setting
$$ \rho(x_1+iy_1,x_2+iy_2) \coloneq \rho(x_1,x_2) + \rho(y_1,y_2) + i\rho(y_1,x_2)- i\rho(x_1,y_2), \quad \text{for every $x_1,x_2,y_1,y_2 \in \Delta$.}$$

\begin{remark}
 We point out that the complexification $\g_\C$ of a real Lie algebra $\g$ is also a real Lie algebra. Moreover, $\g$ is a Lie subalgebra of $\g_\C$. If $\g$ is 2-step nilpotent, then also $\g_\C$ is 2-step nilpotent. If a Lie group $G$ is as in \eqref{def G}, then we define 
 \begin{equation} \label{def Gc}
 G_\C:= (\g_\C,*), \quad \text{where} \quad x*y \coloneq x+y+\tfrac{1}{2}[x,y], \qquad \text{ for all } x,y \in \g_\C.
\end{equation}
Clearly, $G$ is a Lie subgroup of $G_\C$.
\end{remark}

Hereafter, for every $n \in \Z$, we denote by $f_n$ the $n$-th element of the standard Hilbert basis of $L^2([0,1],\C)$:
\begin{gather} 
 f_n \colon [0,1] \to \C, \quad t \mapsto f_n(t):= e^{2\pi i n t}, \label{def_f}\\
 \text { for which }\int^1_0 f_n(t)f_m(t) \, \de t = \delta_{n,-m}, \quad \text{for every $n,m \in \Z$.} \label{def_f_prop}
\end{gather}

\begin{lemma} \label{lemma lemma}
 Let $\g$ be a 2-step nilpotent Lie algebra, and fix an inner product $\rho$ on a subspace $\Delta \subseteq \g$. Let $G_\C$ be as in \eqref{def Gc}, and consider $v \in L^2([0,1],\Delta_\C)$ such that
 \begin{equation} \label{def_xi_lemma}
 v(t) = \sum_{n \in E} c_nf_n(t), 
 \end{equation}
 for some finite subset of non-zero integers $E$, and some $\{c_n\}_{n\in E} \subseteq \Delta_\C$. Then
 \begin{equation} \label{complex_endpoint_1}
 \gamma_v(1) = \sum_{n \in E} \frac{1}{4 \pi i n} [c_n,c_{-n}] \quad \text{and} \quad \int^1_0 \rho(v(t),v(t)) \, \de t
 = \sum_{n \in E} \rho(c_n,c_n).
\end{equation}
\end{lemma}

\begin{proof}
 First, for every $t \in [0,1]$, we compute
 \begin{eqnarray} \label{int_xi_t}
 \int^t_0 v(s) \, \de s &\stackrel{\eqref{def_xi_lemma}}{=}& \int^t_0 \sum_{n \in E} c_nf_n(t) \, \de t = \sum_{n \in E} \frac{c_n}{2 \pi i n}(f_n(t)-1).
 \end{eqnarray}
 where in the second equality we used that $0 \notin E$. In particular, we obtain $\int^1_0 v(t) \, \de t = 0$. Thus, we compute
 \begin{eqnarray*}
 \gamma_v(1) &\stackrel{\eqref{formula integral}}{=}& \int^1_0 v(t) \, \de t + 
 \frac{1}{2} \int^1_0 \left[\int^t_0 v(s) \, \de s, v(t) \right] \, \de t \\ &\stackrel{\eqref{def_xi_lemma}\eqref{int_xi_t}}{=}& \frac{1}{2} \int^1_0 \left[ \sum_{n \in E} \frac{c_n}{2 \pi i n}(f_n(t)-1), \sum_{m \in E} c_mf_m(t)\right] \, \de t \\
 &=& \sum_{m,n \in E} \frac{1}{4\pi i n} [c_n,c_m] \int^1_0 (f_n(t)-f_0(t))f_m(t) \, \de t \\
 &\stackrel{\eqref{def_f_prop}}{=}& \sum_{n,m \in E} \frac{1}{4\pi in} [c_n,c_{m}]\delta_{n,-m} = \sum_{n \in E} \frac{1}{4\pi in} [c_n,c_{-n}]. 
 \end{eqnarray*} 
 We then verified the first identity in \eqref{complex_endpoint}. For the second one, we compute 
 \begin{eqnarray*}
 \int^1_0 \rho(v(t),v(t)) \, \de t &\stackrel{\eqref{def_xi_lemma}}{=}& \int^1_0 \rho\left(\sum_{n \in E} c_nf_n(t),\sum_{m \in E} c_mf_m(t)\right) \, \de t \\
 &=& \sum_{m,n \in E} \rho(c_n,c_m)\int^1_0 f_n(t)\overline{f_m}(t) \, \de t \\
 &\stackrel{\eqref{def_f_prop}}{=}& \sum_{m,n \in E} \rho(c_n,c_m) \delta_{n,m} \, \de t = \sum_{n \in E} \rho(c_n,c_n),
 \end{eqnarray*}
 where we used that $\rho$ is Hermitian and that $\overline{f_m}(t)=f_{-m}(t)$, for every $m \in \Z$ and $t \in [0,1]$.
 \end{proof}

We conclude this section with a consequence of \cref{lemma lemma}.

\begin{corollary} \label{corollary_ortogonal}
 Let $G$ be as in \eqref{def G}, and fix an inner product $\rho$ on a subspace $\Delta \subseteq \g$. Consider $v \in L^2([0,1],\Delta)$ such that
 \begin{equation*} 
 v(t) = \sum_{n \in E} \re(c_nf_n(t)), 
 \end{equation*}
 for some finite subset of positive integers $E$, and some $\{c_n\}_{n\in E} \subseteq (\Delta_\id)_\C$. Then
\begin{equation} \label{complex_endpoint}
 \gamma_v(1) = \sum_{n \in E} \frac{1}{4 \pi n} [\im(c_n),\re(c_n)] \quad \text{and} \quad \int^1_0 \rho(v(t),v(t)) \, \de t = \frac{1}{2} \sum_{n \in E} \rho(c_n,c_n).
\end{equation}
\end{corollary}

\begin{proof}
 In the assumptions of the statement, we get
 \begin{equation} \label{vre}
 v(t)=\sum_{n \in E} \re(c_nf_n(t))=\sum_{n \in E} \frac{c_n}{2}f_n(t) + \sum_{n \in E} \frac{\overline{c_n}}{2}f_{-n}(t).
 \end{equation}
 Therefore, applying \cref{lemma lemma}, we compute
 \begin{eqnarray*}
 \gamma_v(1) &\stackrel{\eqref{complex_endpoint_1}\eqref{vre}}{=}& \sum_{n \in E} \frac{1}{4 \pi i n}\left[\frac{c_n}{2}, \frac{\overline{c_n}}{2}\right] - \sum_{n \in E} \frac{1}{4 \pi i n}\left[\frac{\overline{c_n}}{2},\frac{c_n}{2} \right]\\
 &=& \sum_{n \in E} \frac{1}{8 \pi i n}\left[c_n,\overline{c_n} \right] \\
 &=&\sum_{n \in E} \frac{1}{8 \pi i n} \left(\left[\re(c_n),-i \im(c_n)\right]+ \left[i \im(c_n), \re(c_n)\right]\right) \\
 &=& \sum_{n \in E} \frac{1}{4 \pi n} \left[\im(c_n), \re(c_n)\right].
 \end{eqnarray*}
 We then verified the first identity in \eqref{complex_endpoint}. For the second one, we also apply \cref{lemma lemma} to compute
 \begin{eqnarray*}
 \int^1_0 \rho_\id(v(t),v(t)) \, \de t &\stackrel{\eqref{complex_endpoint_1}\eqref{vre}}{=}& \sum_{n \in E } \frac{1}{4} \rho(c_n,c_n) + \sum_{n \in E } \frac{1}{4} \rho(\overline{c_n},\overline{c_n}) = \frac{1}{2}\sum_{n \in E } \rho(c_n,c_n),
 \end{eqnarray*}
 where we used that $\rho$ is Hermitian to obtain $\rho(\overline{c_n},\overline{c_n})=\rho(c_n,c_n)$.
\end{proof}


 \section{Main results} \label{main}

\subsection{Vertical variations} Our proofs of \cref{main theorem} and \cref{equivalence quasi_isometries intro} both rely on the following fundamental result.

\begin{lemma}[Perturbation technique] \label{new_lemma}
 Let $G$ be a $2$-step nilpotent group as in \eqref{def G} and let $(\Delta,\rho)$ be a sub-Riemannian structure on $G$, with $V$ as in \eqref{def V}. Fix a norm $\lvert \cdot \rvert$ on $[\g,\g]$. Then, there exists $C>0$ such that for every $u \in L^2([0,1],\Delta_\id)$ and every $\zeta \in [\g,\g]$, there exists $v \in L^2([0,1],V)$ such that
 \begin{enumerate}[(i)]
 \item $\gamma_{u+v}(1)=\gamma_u(1)+\zeta$,
	\item $\int^1_0 \rho_\id(v(t),u(t))\, \de t =0$,
	\item $\int^1_0 \rho_\id(v(t),v(t))\, \de t \le C \lvert \zeta \rvert$.
\end{enumerate}
\end{lemma}

\begin{proof}
We set $m \coloneq \dim([V,V])$, and $K > 0$ to be the constant given in \cref{equivalence_norms} for the norm $\lvert \cdot \rvert$ on $[\g,\g]$ and the restriction to $V$ of the inner product $\rho_\id$. Fix $\zeta \in [\g,\g]$. Thanks to \cref{equivalence_norms}, there exist $x_1,\dots,x_{m} \in V$, $y_1,\dots,y_{m} \in V$, and $\alpha_1,\dots,\alpha_{m} \ge 0$, satisfying $\rho_\id(x_k,x_k)=\rho_\id(y_k,y_k)=1$ and $\rho_\id(x_k,y_k)=0$ for every $1 \le k \le m$, such that
\begin{equation} \label{zeta_brakets}
 \zeta = \sum_{k=1}^{m} \alpha_k[x_k,y_k] \quad \text{and} \quad \sum_{k=1}^{m} \alpha_k \le K \lvert \zeta \rvert.
\end{equation}

	Next, we set $E_k \coloneq [(m+2)(k-1)+1, (m+2)k] \cap \Z$, for every $1 \le k \le m$,	and we observe that $\{E_k\}_{k=1}^{m}$ is a family of pairwise disjoint subsets of positive integers. We also set $N \coloneq m^2 + 2m$, so that we get
	\begin{equation} \label{bound_N}
		n \le N, \quad \text{for every $1 \le k \le m$ and $n \in E_k$.}
	\end{equation}
We fix $u \in L^2([0,1],\Delta_\id)$. For every $n \in \Z$ and $\xi \in V_\C$, we define
	\begin{equation*}
		P^n_\xi \coloneq \int^1_0 f_n(t) \rho_\id(\xi,u(t)) \, \de t \in \mathbb{C}, \quad Q^n_\xi \coloneq \int^1_0 f_n(t) \left[ \int^t_0 u(s) \, \de s , \xi \right] \de t \in [V,V]_\mathbb{C},
	\end{equation*}
	with $f_n$ defined as in \eqref{def_f}. For every $1 \le k \le m$, we consider the tuple $\{z_{n,k}\}_{n \in E_k} \subseteq \mathbb{C}$ to be a solution of the following system of equations:
	\begin{equation} \label{system_c}
		\begin{cases}
			\sum_{n \in E_k} z_{n,k}P^n_{y_k + ix_k} = 0, \\
			\sum_{n \in E_k} z_{n,k}Q^n_{y_k + ix_k} = 0, \\
			\sum_{n \in E_k} \tfrac{1}{4 \pi n} \lvert z_{n,k} \rvert^2 = \alpha_k.
		\end{cases}
	\end{equation}
	We claim that such a solution exists. Indeed, a non-trivial solution to the linear part of the system exists since there are $m + 1$ equations and $\lvert E_k \rvert = m + 2$ variables. Such a solution can be normalized so that the last equation is also satisfied. We recall that $\alpha_k \ge 0$.

 Next, for every $1 \le k \le m$, we define $v_k \colon [0,1] \to V$ such that
 \begin{equation} \label{def_vk}
 v_k(t) \coloneq \sum_{n \in E_k} \re\left(z_{n,k}(y_k+ix_k)f_n(t)\right), \quad \text{for every $t \in [0,1]$.} 
 \end{equation}
 In what follows, we show that $v \coloneq \sum_{k=1}^{m} v_k \in L^2([0,1],V)$ satisfies the conditions of the statement.

For proving \emph{(i)}, we first claim that $\gamma_v(1) = \zeta$. Indeed, we observe that the sets $E_k$ are pairwise disjoint, and we compute:
 \begin{eqnarray}
 \gamma_v(1)&\stackrel{\eqref{complex_endpoint}}{=}& \sum_{k=1}^m \sum_{n \in E_k} \frac{1}{4 \pi n} [\im(z_{n,k}(y_k+ix_k)),\re(z_{n,k}(y_k+ix_k))] \notag \\
 &=& \sum_{k=1}^m \sum_{n \in E_k} \frac{1}{4 \pi n} \lvert z_{n,k}\rvert^2 [x_k,y_k] \stackrel{\eqref{system_c}}{=} \sum_{k=1}^m \alpha_k[x_k,y_k] \stackrel{\eqref{zeta_brakets}}{=} \zeta. \label{end_v}
 \end{eqnarray}

Next, we claim that $\gamma_{u+v}(1)=\gamma_u(1)+\gamma_v(1)$. Indeed, we observe that $\int^1_0 v(t) \, \de t = 0$ and we compute
\begin{eqnarray*}
 \int^1_0 \left[\int^t_0 u(s) \, \de s , v(t)\right] \de t &=& \int^1_0 \left[\int^t_0 u(s) \, \de s , \sum_{k=1}^mv_k(t)\right] \de t \\
 &=& \sum_{k=1}^m \sum_{n \in E_k} \int^1_0 \left[\int^t_0 u(s) \, \de s , \re\left(z_{n,k}(y_k+ix_k)f_n(t)\right)\right] \de t \\
 &=& \sum_{k=1}^m \re\left(\sum_{n \in E_k}z_{n,k}\int^1_0 f_n(t) \left[\int^t_0 u(s) \, \de s , y_k+ix_k \right] \de t\right) \\
 &=& \sum_{k=1}^m \re\left(\sum_{n \in E_k} z_{n,k} Q^n_{y_k + ix_k}\right) \stackrel{\eqref{system_c}}{=} 0.
\end{eqnarray*}
Therefore, the hypothesis of \cref{lemma_wedge} are satisfied and we get that 
$$\gamma_{u+v}(1)=\gamma_u(1)+\gamma_v(1)\stackrel{\eqref{end_v}}{=}\gamma_u(1)+\zeta.$$ In particular, \emph{(i)} is verified.

Next, we check \emph{(ii)}, i.e., $\int^1_0 \rho_\id(v(t),u(t))\, \de t =0$. We compute
\begin{eqnarray*}
 \int^1_0 \rho_\id(v(t),u(t))\, \de t &=& \sum_{k=1}^m \left(\int^1_0 \rho_\id(v_k(t),u(t))\, \de t \right) \\ &=& \sum_{k=1}^m \sum_{n \in E_k} \left(\int^1_0 \rho_\id(\re\left(z_{n,k}(y_k+ix_k)f_n(t)\right),u(t))\, \de t \right) \\ &=& \sum_{k=1}^m \re\left(\sum_{n \in E_k} z_{n,k} \int^1_0 f_n(t)\rho_\id(y_k+ix_k,u(t))\, \de t \right) \\ &=& \sum_{k=1}^m \re\left(\sum_{n \in E_k} z_{n,k} P^n_{w_k + iv_k}\right) \stackrel{\eqref{system_c}}{=} 0,
\end{eqnarray*}
where in the third equality we used the fact that $u(t) \in V$ in order to move the operator $\re$ outside the Hermitian product. Thus, \emph{(ii)} is verified.

Finally, to prove \emph{(iii)}, we estimate
\begin{eqnarray*}
 \int^1_0 \rho_\id(v(t),v(t))\, \de t &\stackrel{\eqref{complex_endpoint}}{=}& \frac{1}{2} \sum_{k=1}^m \sum_{n \in E_k} \lvert z_{n,k} \rvert^2 \rho_\id(y_k+ix_k,y_k+ix_k) \\
 &=& \sum_{k=1}^m \sum_{n \in E_k} \lvert z_{n,k} \rvert^2 \le 4 \pi N \sum_{k=1}^m \sum_{n \in E_k} \frac{1}{4\pi n}\lvert z_{n,k} \rvert^2 \\
 &\stackrel{\eqref{system_c}}{=}& 4 \pi N \sum_{k=1}^m \alpha_k \stackrel{\eqref{zeta_brakets}}{\le} 4\pi KN \lvert \zeta \rvert,
\end{eqnarray*}
where in the second equality we used the fact that $\rho_\id(x_k,x_k)=\rho_\id(y_k,y_k)=1$ and $\rho_\id(x_k,y_k)=0$, thus $\rho_\id(y_k+ix_k,y_k+ix_k)=2$. The last statement \emph{(iii)} is then proved by setting $C \coloneq 4\pi KN$. \qedhere
\end{proof}

We include two consequences of \cref{new_lemma}.

\begin{corollary}[Improved Ball-Box Theorem] \label{powerful_lemma}
 Let $(G,d)$ be a simply connected, sub-Riemannian, $2$-step nilpotent Lie group and denote by $\g$ its associated Lie algebra. Fix a norm $\lvert \cdot \rvert$ on $[\g,\g]$. Then, there exists $C>0$ such that
 \begin{equation} \label{pitagora_ballbox}
 d(e,q \cdot \exp(\zeta))^2 \le d(e,q)^2 + C \lvert \zeta \rvert, \quad \text{for every $q \in G$ and $\zeta \in [\g,\g]$.}
 \end{equation}
\end{corollary}
We stress that \eqref{pitagora_ballbox} is an improvement of the triangular inequality combined with the Ball-Box Theorem (see \cite[Theorem~4.3.1]{donne2024metricliegroupscarnotcaratheodory}), which implies the following weaker estimate:
\begin{equation*}
 d(e,q \cdot \exp(\zeta)) \le d(e,q) + C \sqrt{\lvert \zeta \rvert}, \quad \text{for every $q \in G$ and $\zeta \in [\g,\g]$.}
\end{equation*}

\begin{proof}[Proof of \cref{powerful_lemma}] Thanks to \cref{model}, we can assume $G$ to be as in \eqref{def G} and the sub-Riemannian metric $d$ to be associated with a sub-Riemannian structure $(\Delta,\rho)$ on $G$. We define $V$ as in \eqref{def V}, and we set $C_1$ to be the constant given in \cref{new_lemma} for the norm $\lvert \cdot \rvert$ on $[\g,\g]$.

Fix $q \in G$ and $\zeta \in [\g,\g]$. Consider an energy-minimizing curve $\gamma \colon [0,1] \to G$ such that $\gamma(0)=\id$ and $\gamma(1)=q$. Let $u \in L^2([0,1],\Delta_\id)$ denote the control of $\gamma$. Thanks to \cref{new_lemma}, there exists $v \in L^2([0,1],V)$ such that 
\begin{equation*}
 \gamma_{u+v}(1) = \gamma_u(1) + \zeta \stackrel{\eqref{def G}}{=} q * \exp(\zeta),
 \end{equation*}
 together with
 \begin{equation*}
 \int^1_0 \rho_\id(v(t),u(t))\, \de t =0 \quad \text{and} \quad \int^1_0 \rho_\id(v(t),v(t))\, \de t = C_1 \lvert \zeta \rvert.
 \end{equation*}
 Therefore, we estimate
 \begin{eqnarray*}
 d(e,q \cdot \exp(\zeta))^2 &\stackrel{\eqref{distance_square}}{\le}& \int^1_0 \rho_\id(u(t) + v(t),u(t) + v(t))\, \de t \\
 &=& \int^1_0 \rho_\id(u(t),u(t))\, \de t + \int^1_0 \rho_\id(v(t),v(t))\, \de t \\
 &=& d(e,q)^2 + C_1 \lvert \zeta \rvert,
 \end{eqnarray*}
 where in the last inequality we used the fact that $u$ is an energy-minimizing control. The statement is then proved by setting $C\coloneq C_1$.
\end{proof}

Thanks to \cref{powerful_lemma}, we get the following improvement of the classical description of geodesics in $2$-step nilpotent Lie groups \cite[Proposition~7.3.10]{donne2024metricliegroupscarnotcaratheodory}, obtaining in addition a uniform bound on $\zeta$:

\begin{corollary} \label{bounded drift}
 Let $G$ be a $2$-step nilpotent group as in \eqref{def G} and let $(\Delta,\rho)$ be a sub-Riemannian structure on $G$, with $V$ as in \eqref{def V}. Then, there exists $C > 0$ such that, for every energy-minimizing curve $\gamma \colon [0,1] \to G$ with $\gamma(0) = \id$, there are $M \in \mathfrak{so}(V)$, $b \in \ker(M)$, $c \in M(V)$, and $\zeta \in [V,V] \cap \Delta$ such that $\sqrt{\rho_\id(\zeta,\zeta)}\le C$ and
	\begin{equation} \label{equation_geodesic}
 \gamma(t)= x(t) + \frac{1}{2}\int^t_0 [x(s),\dot{x}(s)] \, \de s + t\zeta, \quad \text{for every $t \in [0,1]$.}
	\end{equation}
 where $x(t) \coloneq c - e^{tM}c + tb \in V$, for every $t \in [0,1]$.
\end{corollary}

\begin{proof}
 Fix on $[V,V]$ a norm $\lvert \cdot \rvert$ that extends the restriction of $\sqrt{\rho(\cdot,\cdot)}$ on $\Delta \cap [V,V]$. Let $C_1>0$ be the constant from \cref{powerful_lemma} so that 
 \begin{equation} \label{pitagora_ballbox2}
 d(e,q \cdot \exp(\zeta))^2 \le d(e,q)^2 + C_1 \lvert \zeta \rvert, \quad \text{for every $q \in G$ and $\zeta \in [\g,\g]$.}
 \end{equation}
Fix $q \in G$ and consider an energy-minimizing curve $\gamma \colon [0,1] \to G$ such that $\gamma(0)=\id$ and $\gamma(1)=q$. Let $u \in L^2([0,1],\Delta_\id)$ denote the control of $\gamma$.
From \cite[Proposition~7.3.10]{donne2024metricliegroupscarnotcaratheodory}, there are $M \in \mathfrak{so}(V)$, $b \in \ker(M)$, $c \in M(V)$, and $\zeta \in [V,V] \cap \Delta$ such that 
\begin{equation*}
 \gamma(t)= x(t) + \frac{1}{2}\int^t_0 [x(s),\dot{x}(s)] \, \de s + t\zeta, \quad \text{for every $t \in [0,1]$.}
	\end{equation*}
 where $x(t) \coloneq c - e^{tM}c + tb \in V$, for every $t \in [0,1]$.
We define $u_\infty \in L^2([0,1],V)$ such that, for every $t \in [0,1]$, the value $u_\infty(t)$ is the projection of $u(t)$ on $V$ along $[V,V]$ . We observe that, for every $t \in [0,1]$, we have $\gamma(t)=\gamma_{u_\infty}(t)+t\zeta$. In particular
\begin{equation} \label{point_q_2}
 q = \gamma_{u_\infty}(1) + \zeta = \gamma_{u_\infty}(1) \cdot \exp(\zeta).
\end{equation}
Moreover, since $V$ and $[V,V] \cap \Delta$ are $\rho_\id$-orthogonal, we also have
\begin{equation} \label{energy_gamma_2}
 \int^1_0 \rho_\id(u(t),u(t))\, \de t =\int^1_0 \rho_\id(u_\infty(t),u_\infty(t))\, \de t + \rho_\id(\zeta,\zeta).
\end{equation}
We then estimate
\begin{eqnarray*}
 \rho(\zeta,\zeta) &\stackrel{\eqref{energy_gamma_2}}{=}& \int^1_0 \rho_\id(u(t),u(t))\, \de t - \int^1_0 \rho_\id(u_\infty(t),u_\infty(t))\, \de t \\
 &=& d(e,q)^2 - \int^1_0 \rho_\id(u_\infty(t),u_\infty(t))\, \de t \\
 &\stackrel{\eqref{point_q_2}}{=}& d(e,\gamma_{u_\infty}(1) \cdot \exp(\zeta))^2 - \int^1_0 \rho_\id(u_\infty(t),u_\infty(t))\, \de t \\
 &\stackrel{\eqref{pitagora_ballbox2}}{\le}& d(e,\gamma_{u_\infty}(1))^2 + C_1\lvert \zeta \rvert - \int^1_0 \rho_\id(u_\infty(t),u_\infty(t))\, \de t \\
 &\stackrel{\eqref{distance_square}}{\le}& \int^1_0 \rho_\id(u_\infty(t),u_\infty(t))\, \de t + C_1\lvert \zeta \rvert - \int^1_0 \rho_\id(u_\infty(t),u_\infty(t))\, \de t
 = C_1\lvert \zeta \rvert,
\end{eqnarray*}
where in the second equality we use the fact that $u$ is an energy-minimizing control. Since $\lvert \cdot \rvert$ coincides with $\sqrt{\rho(\cdot,\cdot)}$ on $\Delta \cap [V,V]$, we get $\rho(\zeta,\zeta) \le C_1 \sqrt{\rho(\zeta,\zeta)}$. We conclude that $\sqrt{\rho(\zeta,\zeta)} \le C_1$. The statement is then proved by setting $C \coloneq C_1$.
\end{proof}

\subsection{Proof of theorems} We are finally ready to prove \cref{main theorem}.

\begin{proof}[Proof of \cref{main theorem}]
 Thanks to \cref{model}, we can assume $G$ to be as in \eqref{def G} and the sub-Riemannian metric $d$ to be associated with a sub-Riemannian structure $(\Delta,\rho)$ on $G$ with $V$ as in \eqref{def V}. Let $d_\infty$ be the canonical asymptotic metric of $d$. Thanks to \cref{powerful_lemma}, there exists $C_1>0$ such that 
 \begin{equation} \label{pitagora_ballbox3}
 d_\infty(e,q \cdot \exp(\zeta))^2 \le d_\infty(e,q)^2 + C_1 \lvert \zeta \rvert, \quad \text{for every $q \in G$ and $\zeta \in [\g,\g]$.}
 \end{equation}

 Since the metric $d$ is left-invariant, we assume $p$ in \eqref{main equation} to be the identity element $e$. Fix $q \in G$ and consider an energy-minimizing curve $\gamma \colon [0,1] \to G$ such that $\gamma(0)=\id$ and $\gamma(1)=q$. Let $u \in L^2([0,1],\Delta_\id)$ denote the control of $\gamma$. We define $u_\infty \in L^2([0,1],V)$ such that $u_\infty(t)$ is the projection of $u(t)$ on $V$ parallel to $[V,V]$, for every $t \in [0,1]$. From \cref{bounded drift}, we observe that, for every $t \in [0,1]$, we have $\gamma(t)=\gamma_{u_\infty}(t)+t\zeta$, for some $\zeta \in [\g,\g]$. In particular
\begin{equation} \label{point_q_3}
 q = \gamma_{u_\infty}(1) + \zeta = \gamma_{u_\infty}(1) \cdot \exp(\zeta).
\end{equation}
From \cref{bounded drift}, there exists $C_2$, independent of $\zeta$, such that
\begin{equation} \label{bound zeta}
 \sqrt{\rho(\zeta,\zeta)} \le C_2.
\end{equation}
Since $\gamma_{u_\infty}$ is a $\Delta_\infty$-horizontal curve, we have
\begin{equation} \label{infinity_eta}
 d_\infty(e,\gamma_{u_\infty}(1))^2 \le \int^1_0 \rho_\id(u_\infty(t),u_\infty(t))\, \de t \le \int^1_0 \rho_\id(u(t),u(t))\, \de t,
\end{equation}
where in the second inequality we used the fact that $V$ and $[V,V] \cap \Delta$ are $\rho_\id$-orthogonal. 

We then estimate
\begin{eqnarray*}
 d_\infty(e,q)^2 &\stackrel{\eqref{point_q_3}}{=}& d_\infty(e, \gamma_{u_\infty}(1) \cdot \exp(\zeta))^2 \\
 &\stackrel{\eqref{pitagora_ballbox3}}{\le}& d_\infty(e, \gamma_{u_\infty}(1))^2 + C_1 \sqrt{\rho(\zeta,\zeta)} \\
 &\stackrel{\eqref{infinity_eta}}{\le}& \int^1_0 \rho_\id(u(t),u(t))\, \de t + C_1\sqrt{\rho(\zeta,\zeta)} \\
 &\stackrel{\eqref{bound zeta}}{\le}& d(e,q)^2+ C_1C_2,
\end{eqnarray*}
where in the last inequality we use the fact that $u$ is an energy-minimizing control. Thus, we get a bound on $d_\infty(e,q)^2 - d(e,q)^2$. The bound on $d(e,q)^2 - d_\infty(e,q)^2$ immediately follows from \eqref{obvious bound}. The statement \eqref{main equation} is then proved by setting $C \coloneq C_1C_2$. 

From \eqref{main equation} we get
\begin{equation} \label{big o estimate}
 \left\lvert d_\infty(p,q) - d(p,q) \right\rvert \le \frac{C}{d_\infty(p,q) + d(p,q)} \quad \text{for every $p,q \in G, \, p \neq q$,}
 \end{equation} 
 which proves \eqref{first consequence}. 
\end{proof}

From \cref{main theorem}, we obtain the following consequence:

\begin{corollary} \label{bounded distance cor}
In the assumption of \cref{main theorem}, 
 the identity map $\mathrm{id} \colon (G,d) \to (G,d_\infty)$ is a rough isometry. 
	 \end{corollary}

\begin{proof}
 Clearly, we have 
 \begin{equation*}
 \left\lvert d_\infty(p,q) - d(p,q) \right\rvert \le d_\infty(p,q) + d(p,q) \quad \text{for every $p,q \in G$.}
 \end{equation*}
By combining the previous estimate with \eqref{main equation}, in the form of \eqref{big o estimate}, we obtain, for every $p \neq q$:
\begin{equation*} 
 \left\lvert d_\infty(p,q) - d(p,q) \right\rvert \le \min \Set{\frac{C}{d_\infty(p,q) + d(p,q)},d_\infty(p,q) + d(p,q)} \le \sqrt{C},
\end{equation*}
where in the last inequality we estimated the minimum with the geometric mean. The estimate for $p=q$ is clearly satisfied as well. We conclude that the identity map $(G,d) \to (G,d_\infty)$ is a rough isometry.
\end{proof} 


Before proving \cref{equivalence quasi_isometries intro}, we recall some terminology and basic facts about the abelianization of a metric Lie group. We refer to \cite[Section~10.1]{donne2024metricliegroupscarnotcaratheodory} for details. Hereafter, $\mathrm{Aut}(G)$ denotes the group of Lie group automorphisms of a Lie group $G$.

Let $G$ be a nilpotent Lie group with associated Lie algebra $\g$. The \emph{abelianization} $G_\mathrm{ab}$ of $G$ is the quotient Lie group 
\begin{equation*}
 G_\mathrm{ab} \coloneq G/[G,G].
\end{equation*}
We denote by $\pi_\mathrm{ab}$ the quotient map $\pi_\mathrm{ab} \colon G \to G_\mathrm{ab}$. If $\vphi \in \mathrm{Aut}(G)$, then there exists a unique $\vphi_\mathrm{ab} \in \mathrm{Aut}(G_\mathrm{ab})$ such that $\pi_\mathrm{ab} \circ \vphi = \vphi_\mathrm{ab} \circ \pi_\mathrm{ab}$. We refer to $\vphi_\mathrm{ab}$ as the \emph{abelianization map of $\vphi$.}

If $(G,d)$ is a nilpotent sub-Riemannian Lie group, then there exists a unique Euclidean norm $\lVert \cdot \rVert_\mathrm{ab}$ on $G_\mathrm{ab}$ such that the quotient map
 \begin{equation*}
 \pi_\mathrm{ab} \colon (G,d) \to (G_\mathrm{ab}, \lVert \cdot \rVert_\mathrm{ab})
 \end{equation*}
 is a submetry. We refer to $\lVert \cdot \rVert_\mathrm{ab}$ as the \emph{abelianization norm of $(G,d)$}. 
 
\begin{remark} \label{equality automorphism}
 If $G$ is a simply connected nilpotent Lie group with associated Lie algebra $\g = V \oplus [\g,\g]$ and $\vphi,\vphi' \in \mathrm{Aut}(G)$ are such that $\vphi_\mathrm{ab} = \vphi'_\mathrm{ab}$ and $\de_\id \vphi(V) = \de_\id \vphi'(V)$, then $\vphi = \vphi'$. This is a consequence of the fact that, in each nilpotent Lie algebra $\g$, every subspace $V \subseteq \g$ complementary to $[\g,\g]$ is bracket-generating; see \cite[Exercise~10.5.8]{donne2024metricliegroupscarnotcaratheodory}.
\end{remark}

\begin{lemma} \label{same projection}
 Let $G$ be a simply connected, $2$-step nilpotent Lie group, with associated Lie algebra $\g$, equipped with sub-Riemannian structures $(\Delta,\rho)$ and $(\Delta',\rho')$ such that
 \begin{equation} \label{min bracket generating}
 \Delta_e \oplus [\g,\g] = \g \quad \text{and} \quad \Delta'_e \oplus [\g,\g] = \g.
 \end{equation}
 Denote by $d$ and $d'$ the sub-Riemannian metrics associated with $(\Delta,\rho)$ and $(\Delta',\rho')$, respectively. Also, denote by $\lVert \cdot \rVert_\mathrm{ab}$ and $\lVert \cdot \rVert'_\mathrm{ab}$ the abelianization norms of $(G,d)$ and $(G,d')$, respectively.
 If $\lVert \cdot \rVert_\mathrm{ab}=\lVert \cdot \rVert'_\mathrm{ab}$, there exists a Lie group morphism $\vphi \in \mathrm{Aut}(G)$ such that the following hold
 \begin{enumerate}[(i)]
 \item $\vphi:(G,d)\to(G,d')$ is an isometry;
 \item $\pi_\mathrm{ab} \circ \vphi = \pi_\mathrm{ab}$, in particular there exists a linear map $L \colon G_\mathrm{ab} \to [\g,\g]$ such that
 \begin{equation*}
 \vphi(p) = p \cdot \exp\left(L(\pi_\mathrm{ab}(p))\right), \quad \text{for every $p \in G$.}
 \end{equation*}
 \end{enumerate}
\end{lemma}
\begin{proof}
By~\eqref{min bracket generating}, there is a linear map $L:\Delta_e\to[\g,\g]$ such that $\Delta'_e$ is the graph of $L$, that is, 
\[
\Delta'_e = \{(x,Lx)\in \Delta_e \oplus [\g,\g] : x\in\Delta_e \} .
\]
Define $f:\g\to\g$ by $f(x,y) = (x,y+Lx)$ for all $x\in\Delta_e$ and $y\in[\g,\g]$.
Since $f$ is a linear isomorphism with $f([\g,\g])\subseteq[\g,\g]$, and since $\g$ has step 2, 
then $f$ is a Lie algebra automorphism.
Let $\vphi \coloneq \exp \circ f \circ \exp^{-1}$ be the corresponding Lie group automorphism.
Clearly $\vphi$ satisfies (ii).

Notice that the abelianization map $\pi_\mathrm{ab}:\g\to\g/[\g,\g]$ is a linear isomorphism when restricted to $\Delta_e$ and $\Delta'_e$.
It is an easy fact that these restrictions 
$(\Delta_e,\sqrt{\rho_e})\to (\g/[\g,\g],\lVert \cdot \rVert_\mathrm{ab})$
and $(\Delta'_e,\sqrt{\rho'_e})\to (\g/[\g,\g],\lVert \cdot \rVert'_\mathrm{ab})$
of $\pi_\mathrm{ab}$
are submetries (see \cite[Definition~10.1.1]{donne2024metricliegroupscarnotcaratheodory}) and thus isometries.
Therefore,
\[
\sqrt{\rho_e((x,0),(x,0))} = \lVert \pi_\mathrm{ab}(x,0)\rVert_\mathrm{ab}
= \lVert \pi_\mathrm{ab}(x,Lx) \rVert'_\mathrm{ab}
= \sqrt{\rho'_e((x,Lx),(x,Lx))} .
\]
We conclude that the restriction $f:(\Delta_e,\rho_e)\to(\Delta'_e,\rho'_e)$ is an isometry and therefore $\vphi:(G,d)\to(G,d')$ is an isometry.
\end{proof}


Instead of \cref{equivalence quasi_isometries intro}, we actually prove the following sharper statement.

 \begin{theorem} \label{equivalence quasi_isometries}
 Let $G$ be a simply connected, $2$-step nilpotent Lie group equipped with two sub-Riemannian structures $(\Delta,\rho)$ and $(\Delta',\rho')$ and denote by $d$ and $d'$ the sub-Riemannian metrics corresponding to $(\Delta,\rho)$ and $(\Delta',\rho')$, respectively. The following are equivalent:
 \begin{enumerate}[(i)]
 \item the two metrics are asymptotic, i.e., $\frac{d'(p,q)}{d(p,q)} \to 1$ as $d(p,q) \to \infty$;
 \item the abelianization norms $\lVert \cdot \rVert_\mathrm{ab}$ and $\lVert \cdot \rVert'_\mathrm{ab}$ of $(G,d)$ and $(G,d')$, respectively, coincides;
 \item the identity map $\mathrm{id} \colon (G,d) \to (G,d')$ is a rough isometry, i.e., there exists $C>0$ such that $$\lvert d(p,q) - d'(p,q) \rvert \le C, \quad \text{for every $p,q \in G$.}$$
 \end{enumerate}
 \end{theorem}

\begin{proof}[Proof of \cref{equivalence quasi_isometries}]
 The equivalence $(i)\Leftrightarrow(ii)$ is known to hold for simply connected sub-Finsler nilpotent Lie groups; see \cite[Proposition~3]{MR3153949}. The implication $(iii)\Rightarrow(i)$ is immediate. We now prove that $(i, ii)\Rightarrow(iii)$.

 Denote by $\g$ the Lie algebra of $G$. Thanks to \cref{bounded distance cor}, it is not restrictive to assume that $d$ and $d'$ are canonical asymptotic metrics, for which \eqref{min bracket generating} is satisfied. From \cref{same projection}, there exists a Lie group morphism $\vphi \in \mathrm{Aut}(G)$ and a linear map $L \colon G_\mathrm{ab} \to [\g,\g]$ such that
 \begin{align}
 d'(p,q) &= d(\vphi(p),\vphi(q)), \quad \text{for every $p,q \in G$}, \label{phi d}\\
 \vphi(p) &= p \cdot \exp\left(L(\pi_\mathrm{ab}(p))\right), \quad \text{for every $p \in G$.} \label{phi L}
 \end{align}
 Thanks to \cref{powerful_lemma}, there exists $C_1 > 0$ such that 
 \begin{equation} \label{pitagora_ballbox4}
 d(e,q \cdot \exp(\zeta))^2 \le d(e,q)^2 + C_1 \lvert \zeta \rvert, \quad \text{for every $q \in G$ and $\zeta \in [\g,\g]$.}
 \end{equation}
 We want to bound $\lvert d(p,q) - d'(p,q)\rvert$, since $d$ and $d'$ are both left-invariant, it is not restrictive to consider $p=\id$. Fix $q \neq \id$ and a norm $\lvert \cdot \rvert$ on $[\g,\g]$. We estimate
 \begin{eqnarray*}
 d'(\id,q)^2 &\stackrel{\eqref{phi d}}{=}& d(\vphi(\id),\vphi(q))^2 \\
 &\stackrel{\eqref{phi L}}{=}& d(\id,q \cdot \exp\left(L(\pi_\mathrm{ab}(q))\right))^2 \\
 &\stackrel{\eqref{pitagora_ballbox4}}{\le}& d(\id,q)^2 + C_1\lvert L(\pi_\mathrm{ab}(q)) \rvert \\
 &\le& d(\id,q)^2 + C_1\lVert L \rVert \lvert \pi_\mathrm{ab}(q) \rvert \\
 &\le& d(\id,q)^2 + C_1 \lVert L \rVert d(\id,q),
 \end{eqnarray*}
where $\lVert L \rVert$ denotes the operator norm of $L \colon (G_\mathrm{ab},\lVert \cdot \rVert_\mathrm{ab}) \to ([\g,\g],\lvert \cdot \rvert)$ and in the last inequality we used the fact that $\pi_\mathrm{ab}$ is a submetry. We obtain
\begin{equation} 
 d'(\id,q)^2 \le d(\id,q)^2 + C_1 \lVert L \rVert d(\id,q) \le (d(\id,q)^2 + C_1 \lVert L \rVert)^2,
\end{equation}
and therefore $d'(e,q) - d(e,q) \le C_1 \lVert L \rVert$.
Similarly, thanks to \cref{powerful_lemma}, there exists $C_2 > 0$ such that 
 \begin{equation*} 
 d'(e,q \cdot \exp(\zeta))^2 \le d'(e,q)^2 + C_2 \lvert \zeta \rvert, \quad \text{for every $q \in G$ and $\zeta \in [\g,\g]$.}
 \end{equation*}
 An analogous argument shows that $d(\id,q) - d'(\id,q) \le C_2 \lVert L \rVert$. The statement is then proved by setting $C \coloneq \max\set{C_1,C_2} \lVert L \rVert$.
\end{proof}

\section{Consequences, examples, and sharpness of results} \label{consequences}

\subsection{Rate of convergence to the asymptotic cone} \cref{main theorem} provides a finer estimate on the rate of convergence of each simply connected $2$-step nilpotent sub-Riemannian Lie group $(G,d)$ to its asymptotic cone $(\widehat{G},\widehat{d})$. We refer to \cite[Chapters 7,8]{MR1835418} for the notion of Gromov-Hausdorff distance between metric spaces, Gromov-Hausdorff limit, and asymptotic cone of a metric space.

\begin{proposition} \label{prop_rate_convergence}
 Let $(G,d)$ be a simply connected $2$-step nilpotent sub-Riemannian Lie group and denote by $(\widehat{G},\widehat{d})$ its asymptotic cone. Then
 \begin{equation} \label{rate_convergence}
 d_{GH}\left((G,\tfrac{1}{n}d),(\widehat{G},\widehat{d})\right) = O(n^{-1}), \quad \text{as $n \to\infty$,}
 \end{equation}
 where $d_{GH}$ denotes the Gromov-Hausdorff distance.
\end{proposition}

\begin{proof}
 Let $(G,d)$ be as in the statement and denote by $(G,d_\infty)$ its canonical asymptotic metric, defined in \eqref{asymptotic metric}. As discussed in \cref{section_asymptotic_metric}, $(G,d_\infty)$ is isometric to the asymptotic cone $(\widehat{G},\widehat{d})$ of $(G,d)$. From \cref{bounded distance cor}, there exists $C>0$ such that
\begin{equation*}
 \left\lvert \frac{1}{n}d_\infty(p,q) - \frac{1}{n}d(p,q) \right\rvert \le \frac{C}{n} \quad \text{for every $p,q \in G$, and $n \in \N$.}
\end{equation*}
Thus, we infer (see \cite[Corollary 7.3.28]{MR1835418}) that
\begin{equation} \label{gh_inequality1}
 d_{GH}\left((G,\tfrac{1}{n}d),(G,\tfrac{1}{n}d_\infty)\right) \le \frac{2C}{n}.
 \end{equation}
We recall that $(G,d_\infty)$ is a Carnot metric and therefore it is self-similar, i.e., $(G,d_\infty)$ is isometric to $(G,\lambda d_\infty)$ for every $\lambda > 0$. We conclude that $(G,\tfrac{1}{n}d_\infty)$ is isometric to $(\widehat{G},\widehat{d})$. We have proved \eqref{rate_convergence}. 
\end{proof}

The latter result improves similar known estimates for simply connected, $2$-step nilpotent sub-Finsler Lie groups. In \cite{MR3153949}, it is proved that
\begin{equation*}
 d_{GH}\left((B(n),\tfrac{1}{n}d),(\widehat{B}(1),\widehat{d})\right) = O(n^{-\frac{1}{2}}), \quad \text{as $n \to\infty$,} 
 \end{equation*}
 where $B(n) \coloneq \set{p \in G \colon d(0,p) \le n}$ and $\widehat{B}(1) \coloneq \set{p \in \widehat{G} \colon \widehat{d}(0,p) \le 1}$. The authors provided an example of a $2$-step nilpotent sub-Finsler Lie group $(G,d)$ for which the exponent $-\tfrac{1}{2}$ is sharp, and in particular it is not $(1,C)$-quasi-isometric to its asymptotic cone.

\subsection{The non-validity for 3-step Riemannian metrics}

 We show that the consequence of \cref{main theorem} cannot be extended to sub-Riemannian nilpotent Lie groups of higher step. We provide an example of a Riemannian metric $d$ on the Engel group that is not at finite distance from its canonical asymptotic metric.

 \begin{example}
 The \emph{Engel group} $\E$ is the simply connected Lie group whose associated Lie algebra $\g$ admits a basis $\set{X_1,X_2,Y,Z}$ such that the only non-trivial bracket relations between elements of the basis are
 \begin{equation*}
 [X_1,X_2]=Y, \qquad [X_1,Y]=Z.
 \end{equation*}
 We consider the left-invariant Riemannian metric $d$ on $\E$ such that $\set{X_1,X_2,Y,Z}$ is an orthonormal frame for the metric tensor. Since $[\g,\g]=\spn\set{Y,Z}$, we get that the canonical asymptotic metric $d_\infty$ associated with $d$ is the metric $d_\infty$ induced by the sub-Riemannian structure $(\Delta,\rho_{|\Delta})$, where
 \begin{equation*}
 \Delta = \bigcup_{p \in \E} \spn\set{\de L_p(X_1), \de L_p(X_2)}.
 \end{equation*}
 We identify $\E$ and $\R^4$ with the global diffeomorphism
 \begin{equation*}
 \Phi \colon \R^4 \to \E, \quad (x_1,x_2,y,z) \mapsto \exp(x_1X_1+x_2X_2+yY+zZ).
 \end{equation*}
 With this identification, we define $p_n \coloneq (0,n,0,\sqrt{n}) \in \E$ for every $n \in \N$. By considering the straight line $\gamma_n(t)\coloneq(0,tn,0,t\sqrt{n})$, we get the estimate
 \begin{equation} \label{d_riemannian}
 d(e,p_n) \le \sqrt{n^2 + n} \le n + \tfrac{1}{2}.
 \end{equation}
 We stress that $(\E,d_\infty)$ is self-similar, in particular
 \begin{equation} \label{dilation_engel}
 d_\infty(e,(x_1,x_2,y,z)) = \lambda d_\infty(e,(\lambda^{-1} x_1, \lambda^{-1} x_2, \lambda^{-2} y, \lambda^{-3} z)), \quad \text{for every $(x_1,x_2,y,z) \in \R^4, \, \lambda \ge 0$.}
 \end{equation}
 Moreover, from \cite[Lemma~4.1]{MR3906166}, there exists $C>0$ such that
 \begin{equation} \label{cusp_engel}
 d_\infty(e,(0,1,0,\varepsilon)) \ge 1 + C\varepsilon^{1/3}, \quad \text{for every $\varepsilon \in [0,1)$.}
 \end{equation}
 Therefore, we get the estimate
 \begin{equation*}
 d_\infty(e,p_n) \stackrel{\eqref{dilation_engel}}{=} n d_\infty(e,(0,1,0,n^{-\frac{5}{2}})) \stackrel{\eqref{cusp_engel}}{\ge} n (1 + Cn^{-\frac{5}{6}})
 = n + Cn^{\frac{1}{6}}.
 \end{equation*}
 By comparing it with \eqref{d_riemannian}, we get
 \begin{equation*}
 \lvert d_\infty(e,p_n)-d(e,p_n)\rvert \ge Cn^{\frac{1}{6}} - \tfrac{1}{2} \to \infty, \quad \text{as $n \to \infty$.}
 \end{equation*}

 \end{example}

 \subsection{Volume growth of metric balls} In the study of groups of polynomial growth, a key aspect is to determine the volume asymptotics of metric balls $B_r$. The leading term of this expansion for nilpotent groups was first characterized in \cite{MR741395}. The error term, however, is far from being understood, even in the class of $2$-step nilpotent metric groups, see \cite{MR3153949}. Nevertheless, we point out the following general argument.
 
 \begin{remark} \label{error_term}
 Let $(G,d, \vol)$ be a metric measure space such that for some $\id\in G$ and $C,Q, r_0>0$ we have
 \begin{equation}\label{homog_meas}
 \vol(B(\id, r))=Cr^Q, \quad \text{for every $r \ge r_0$.}
 \end{equation}
 Consider on $G$ a distance function $d'$ such that for some function $f(r) = O(r)$ as $r \to \infty$ we have
 \begin{equation} \label{error_distance}
 \lvert d(\id,p) - d'(\id,p) \rvert \le f(d(\id,p)), \quad \text{for every $p \in G$}.
 \end{equation}
 Denote by $B(r)$ and $B'(r)$ the balls of radius $r$, centered at $\id$, for the distance functions $d$ and $d'$, respectively. We claim that
 \begin{equation*}
 \vol(B'(r)) = Cr^Q + O(r^{Q-1}f(r)), \quad \text{as $r\to \infty$.}
 \end{equation*}
 Indeed, from \eqref{error_distance}, we get
 \begin{equation*}
 B(r-f(r)) \subseteq B'(r) \subseteq B(r + f(r)), \quad \text{for every $r \ge 0$,},
 \end{equation*}
 and consequently from \eqref{homog_meas} we deduce 
 \begin{equation*}
 Cr^Q + O(r^{Q-1}f(r)) = C(r-f(r))^Q \le \vol(B'(r)) \le C(r+f(r))^Q = Cr^Q + O(r^{Q-1}f(r)),
 \end{equation*}
 where in the equalities we have used $f(r) = O(r)$, as $r \to \infty$. 
 \end{remark}

 Our result \cref{main theorem}, together with \cref{error_term}, provides a fine estimate for this error term in the class of sub-Riemannian $2$-step nilpotent Lie groups.

 \begin{theorem} \label{asymptotic_sphere}
 Let $(G,d)$ be a simply connected $2$-step nilpotent sub-Riemannian Lie group and fix a Haar measure $\vol$ on $G$. For every $r \ge 0$, denote by $B(r) \subseteq G$ the metric ball of radius $r$, centered at the identity element of $G$. Then the following holds:
 \begin{enumerate}[(i)]
 \item there exists $C > 0$ such that
 \begin{equation*}
 \vol(B(r)) = C r^Q + O(r^{Q-2}), \quad \text{as $r \to \infty$,}
 \end{equation*}
 where $Q$ is the Hausdorff dimension of the asymptotic cone of $(G,d)$,
 \item if $\vol(B(r)) = C r^Q + o(r^{Q-2})$, as $r \to \infty$, then the metric Lie group $(G,d)$ coincides with its canonical asymptotic metric. In particular $\vol(B(r)) = C r^Q$.
 \end{enumerate}
 \end{theorem}

 \begin{proof}
 Let $d_\infty$ be the canonical asymptotic metric of $(G,d)$. 
 We denote by $B_\infty(r)$ the metric ball of radius $r$ with respect to the distance function $d_\infty$, centered at the identity element of $G$. 

Note that, being $(G,d_\infty)$ a Carnot group,
 we have 
 \begin{equation} \label{volume_spheres}
     \vol(B_{\infty}(r))=Cr^Q, \quad \text{for every $r \ge 0$,}
 \end{equation}
 where $Q$ is the Hausdorff dimension of $(G,d_\infty)$ and $C \coloneq \vol\left(B_{\infty}(1)\right)$.
 Moreover, by \cref{main theorem}, we have \eqref{first consequence}. 
 Hence, we are in the setting of \cref{error_term} with $f(r)= O(1/r)$. 
 Therefore, we conclude 
 \begin{equation*}
 \vol(B(r)) = Cr^Q + O(r^{Q-1}\cdot r^{-1}) = Cr^Q + O(r^{Q-2}),
 \end{equation*}
 We thus proved (i).
 


 To prove (ii), assume that $(G,d)$ does not coincide with its canonical asymptotic metric $(G,d_\infty)$. Denote by $\g$ the Lie algebra of $G$, and by $(\Delta,\rho)$ and $(\Delta_\infty,\rho)$ the sub-Riemannian structures of $(G,d)$ and $(G,d_\infty)$, respectively. Since $d$ is not a canonical asymptotic metric, there exists $w \in \Delta \cap [\g,\g]$ such that $\rho(w,w)=1$. Let $W$ be the one-parameter subgroup generated by $w$. Since $W$ is a normal subgroup of $G$, we have the manifold identification 
 \begin{equation*}
 G \cong G/W \times W.
 \end{equation*}
 We recall from \cite[Proposition~7.1.9]{donne2024metricliegroupscarnotcaratheodory} that there exists a sub-Finsler metric $d'$ on $G/W$ such that the projection map $\pi \colon (G,d_\infty) \to (G/W,d')$ is a submetry. We denote by $B'(r)$ the metric ball of radius $r$ with respect to the distance function $d'$, centered at the identity element of $G/W$.

 We write the Haar measure $\vol$ on $G$ as the product measure of a Haar measure $\mu$ on $G/W$ and the Lebesgue measure on $W = \set{\exp(tw) \colon t \in \R} \cong \R$. We point out that, since $(G,d_\infty)$ is self-similar, then $(G/W, d')$ is self-similar as well, and that its Hausdorff dimension is $Q-2$; see \cite[Corollary~4.3.6]{donne2024metricliegroupscarnotcaratheodory}. Therefore, there exists $C'>0$ such that
 \begin{equation} \label{volume_proj}
 \mu \left(B'(r) \right) = C' r^{Q-2}, \quad \text{for every $r \ge 0$.}
 \end{equation}

 Next, for every $k,r \ge 0$, we define
 $E_{k,r} \coloneq \set{ p + \lambda w \colon p \in B_\infty\left( r - \tfrac{k^2}{r}\right), \, \lambda \in [-k,k] } \subseteq G$, and we claim that
 \begin{equation} \label{inclusion_spheres}
 E_{k,r} \subseteq B(r), \quad \text{for every $k \ge 0$ and $r \ge k$.}
 \end{equation}
 Indeed, fix $p \in B_\infty\left( r - \tfrac{k^2}{r}\right)$ and $\lambda \in [-k,k]$. Consider a $\Delta_\infty$-horizontal curve $\gamma_u \colon [0,1] \to G$, from $\id$ to $p$, such that $\int^1_0 \rho(u,u) \, \de t \le \left(r - \frac{k^2}{r}\right)^2$. Then $\gamma_{u+\lambda w}$ is a $\Delta$-horizontal curve and
 \begin{align*}
 \int^1_0 \rho(u+\lambda w, u+\lambda w) \, \de t = \int^1_0 \rho(u,u) \, \de t + \lambda^2 \le \left(r - \frac{k^2}{r}\right)^2 + k^2 = r^2 - k^2 + \frac{k^4}{r^2} \le r^2,
 \end{align*}
 where in the first equality we used \eqref{def V}, and in the last inequality we used that $r \ge k$.
 We conclude that $p+\lambda w = \gamma_{u+\lambda w}(1) \in B(r)$, and the inclusion \eqref{inclusion_spheres} is proved.

 We set $k \coloneq \frac{C'}{QC}$. For $r \ge k^2$, we estimate
 \begin{eqnarray*}
 \vol(B(r)) \stackrel{\eqref{inclusion_spheres}}{\ge} \vol(E_{k,r}) &=& \vol\left(B_\infty\left( r - \tfrac{k^2}{r}\right)\right) + 2k\mu\left(B'\left( r - \tfrac{k^2}{r}\right)\right) \\
 &\stackrel{\eqref{volume_spheres}\eqref{volume_proj}}{=}& C\left( r - \tfrac{k^2}{r}\right)^Q + 2kC'\left( r - \tfrac{k^2}{r}\right)^{Q-2} \\
 &=& C(r^Q - Qk^2 r^{Q-2} + o(r^{Q-2})) + 2kC'(r^{Q-2} + o(r^{Q-2})) \\
 &=& Cr^Q + k(2C' - kQC)r^{Q-2} + o(r^{Q-2}) \\
 &=& Cr^Q + kC'r^{Q-2} + o(r^{Q-2}),
 \end{eqnarray*}
 where in the first equality we used the definition of $E_{k,r}$ and we wrote $\vol$ as the product measure of $\mu$ with the Lebesgue measure on $W$. We infer that $\vol(B(r)) - Cr^Q \notin o(r^{Q-2})$ and the proof of (ii) is concluded.
 \end{proof}

We point out that \cref{asymptotic_sphere} shows that the estimate \eqref{main equation} is sharp in the following sense.

\begin{corollary} \label{sharpness}
 Let $(G,d)$ be a sub-Riemannian, $2$-step nilpotent Lie group, and denote by $d_\infty$ its canonical asymptotic metric. If
 \begin{equation*}
 \lvert d_\infty(p,q)^2 - d(p,q)^2 \rvert \to 0, \quad \text{as $d(p,q) \to \infty$.}
 \end{equation*}
 then $d=d_\infty$.
\end{corollary}

 \begin{proof} From the assumption, we get 
 $ \lvert d_\infty(\id,p) - d(\id,p) \rvert = o(1/d(\id,p)) $, as $d(\id,p)\to \infty$.
 By \cref{error_term},
we get
 \begin{equation*}
 \vol(B(r)) = C r^Q + o(r^{Q-2}), \quad \text{as $r\to \infty$.}
 \end{equation*}
 Thanks to (ii) in \cref{asymptotic_sphere}, we infer that $d=d_\infty$.
 \end{proof}

\subsection{The non-validity for Finsler metrics in the Heisenberg group}
 We show that both consequences of \cref{asymptotic_sphere} do not extend to the setting of 2-step nilpotent sub-Finsler Lie groups by providing two examples of Finsler metrics on the Heisenberg group.
 
 \begin{definition}
 The \emph{Heisenberg Lie algebra} is the Lie algebra $\mathfrak{h}$ admitting a basis $\set{X,Y,Z}$ such that the only non-trivial bracket relation between elements of the basis is $[X,Y]=Z$. The \emph{Heisenberg group} is the Lie group $\H \coloneq (\mathfrak{h},*)$, with $*$ as in \eqref{def G}. 
 The \emph{sub-Riemannian Heisenberg group} $(\H,d)$ is the sub-Riemannian Lie group such that $\set{X,Y}$ is an orthonormal frame of the metric tensor $\rho$ defined on the horizontal distribution $\Delta$, with $\Delta_\id = \spn\set{X,Y}$.
 \end{definition}

 For a thorough introduction to Finsler and sub-Finsler metrics in nilpotent Lie groups, and to the sub-Riemannian Heisenberg group, we refer to \cite{donne2024metricliegroupscarnotcaratheodory}.
 

\begin{proposition} \label{prop_linf}
 Let $(\H,d)$ be the sub-Riemannian Heisenberg group with sub-Riemannian structure $(\Delta,\rho)$. For $v \in \mathfrak{h}$, we define the norm
 \begin{equation} \label{linf_norm}
 \lVert v \rVert \coloneq \max \Set{ \sqrt{\rho(v_h,v_h)}, \lvert \lambda \rvert}, 
 \end{equation}
 where we write $v = v_h + \lambda Z$, with $v_h \in \spn \set{X,Y}$ and $\lambda\in \R$. Denote by $d_F$ the Finsler distance function associated with the Finsler norm \eqref{linf_norm}, and by $B_F(r)$ the ball of radius $r$ with respect to the distance function $d_F$, centered at the identity element of $\H$. Fix a Haar measure $\vol$ on $\H$, then there exists $C,D > 0$ such that
 \begin{equation} \label{polinomial inf}
 \vol \left(B_F(r)\right) = Cr^4 + Dr^3, \quad \text{for every $r \ge 0$.}
 \end{equation}
\end{proposition}

\begin{proof}
 Denote by $B(r)$ the ball of radius $r$ with respect to the sub-Riemannian distance function $d$, centered at the identity element $\id$ of $\H$. We claim that, for every $r \ge 0$, we have 
 \begin{equation} \label{equality_balls}
 B_F(r)= \set{p+ \lambda Z \,\colon \, p \in B(r), \, \lambda \in [-r,r]}.
 \end{equation}
 Indeed, fix $p \in B(r)$ and $\lambda \in [-r,r]$. Consider an energy-minimizing, $\Delta$-horizontal curve $\gamma_u \colon [0,1] \to \H$ from $\id$ to $p$. We get that $\sqrt{\rho(u(t),u(t))}\equiv c$ for some constant $c \le r$. We consider the control $\widetilde{u} \colon [0,1] \to \mathfrak{h}$ defined as $\widetilde{u}(t)=u(t) + \lambda Z$. From \cref{lemma integral}, we get that $\gamma_{\tilde{u}}(1) = p + \lambda Z$. Moreover
 \begin{equation*}
 \int^1_0 \lVert u(t) \rVert \, \de t = \int^1_0 \max \Set{\sqrt{\rho(u(t),u(t))}, \lvert \lambda \rvert} \, \de t \le \max \Set{c, \lvert \lambda \rvert}\le r.
 \end{equation*}
 This proves that $\set{p+zZ \,\colon \, p \in B(r), \, \lambda \in [-r,r]} \subseteq B_F(r)$. For the opposite inclusion, fix $q \in B_F(r)$ and consider a curve $\gamma_u \colon [0,1] \to \H$ such that $\int^1_0 \lVert u(t) \rVert \, \de t \le r$. 
 For every $t \in [0,1]$, we write $u(t) = u_h(t) + \lambda(t) Z$. In particular, $\gamma_{u_h}$ is a $\Delta$-horizontal curve and 
 \begin{equation*}
 \int^1_0 \sqrt{\rho(u_h(t),u_h(t))} \, \de t \le \int^1_0 \lVert u(t) \rVert \, \de t \le r,
 \end{equation*}
 which implies $\gamma_{u_h}(1) \in B(r)$. From \cref{lemma integral}, we get $q = \gamma_{u_h}(1) + \left(\int^1_0 \lambda(t) \de t\right)Z$. The proof of \eqref{equality_balls} is concluded by estimating
 \begin{equation*}
 \left\lvert \int^1_0 \lambda(t) \, \de t \right\rvert \le \int^1_0 \lvert \lambda(t) \rvert \, \de t \le \int^1_0 \lVert u(t)\rVert \, \de t \le r.
 \end{equation*}

 Next, we recall that the Haar measure $\vol$ on $\H$ is a multiple of the Lebesgue measure $\mathscr{L}^3$ on $\mathfrak{h} = \spn\set{X,Y,Z} \cong \R^3$. Moreover, the projection map $\pi \colon \mathfrak{h} \to \spn\set{X,Y}$, parallel to $Z$, is a submetry between $(\H,d)$ and the Euclidean distance on $\spn\set{X,Y} \cong \R^2$; see \cite[Proposition~7.1.9]{donne2024metricliegroupscarnotcaratheodory}. Finally, for $r \ge 0$, we compute
 \begin{eqnarray*}
 \vol(B_F(r)) &\stackrel{\eqref{equality_balls}}{=}& 
 \vol \left( \set{p+\lambda z \,\colon \, p \in B(r), \, \lambda \in [-r,r]} \right) \\
 &=& \vol(B(r)) + 2r \mathscr{L}^2(\pi(B(r))) \\
 &=& Cr^4 + Dr^3, 
 \end{eqnarray*}
 for some $C,D>0$.
\end{proof}

\begin{corollary} \label{no bounded carnot}
 Let $(\H,d_F)$ be the Finsler Heisenberg group defined as in \cref{prop_linf}. Then there are no Carnot metrics $d$ on $\H$ such that \eqref{no vanishing} holds.
\end{corollary} 
\begin{proof}
 Assume that there exists a Carnot metric $d$ such that $\lvert d(p,q) - d_F(p,q) \rvert \to 0$, as $d(p,q) \to \infty$.
 Then, by \cref{error_term},
 we get
 \begin{equation*}
 \vol(B_F(r)) = C r^4 + o(r^3), \quad \text{as $r\to \infty$,}
 \end{equation*}
 which is in contradiction with \eqref{polinomial inf}.
\end{proof}

\begin{proposition} \label{prop_ellone}
 Let $(\H,d)$ be the sub-Riemannian Heisenberg group with sub-Riemannian structure $(\Delta,\rho)$. For $v \in \mathfrak{h}$, we define the norm
 \begin{equation} \label{lone_norm}
 \lVert v \rVert \coloneq \sqrt{\rho(v_h,v_h)}+ \lvert \lambda \rvert, 
 \end{equation}
 where we write $v = v_h + \lambda Z$, with $v_h \in \spn \set{X,Y}$ and $\lambda\in\R$. Denote by $d_F$ the Finsler distance function associated with the Finsler norm \eqref{lone_norm}, and by $B_F(r)$ the ball of radius $r$ with respect to the distance function $d_F$, centered at the identity element of $\H$. Fix a Haar measure $\vol$ on $\H$, then there exists $r_0 > 0$ such that
 \begin{equation*}
 d(p,q) = d_F(p,q), \quad \text{for every $p,q \in \H$ such that $d(p,q)\ge r_0$.}
 \end{equation*}
 In particular, there exists $C>0$ such that $\vol(B_F(r))=Cr^4$ for every $r \ge r_0$.
\end{proposition}

\begin{proof}
 Since both metrics $d$ and $d_F$ are left-invariant, it is not restrictive to assume $p=\id$.
 Let us consider the projection map $\pi \colon \H \to \spn\set{X,Y}$, parallel to $Z$. From the complete description of geodesics in the sub-Riemannian Heisenberg group (see \cite[Section~2.4.1]{donne2024metricliegroupscarnotcaratheodory}), there exists $c>0$ such that for every $q \in \H$ and every $d$-minimizing curve $\gamma \colon [0,1] \to \H$, from $\id$ to $q$, we get
 \begin{equation} \label{existence_time}
 \lVert \pi(\gamma(t_0)) \rVert = c \cdot d(\id,q), \quad \text{for some $0 \le t_0 \le 1$.}
 \end{equation}
 Fix $q \in \H$ such that $d_F(\id,q) \ge 3c^{-1} + d(\id,Z)^2$, and consider a $d_F$-minimizing curve $\gamma_u \colon [0,1] \to \H$ from $e$ to $q$. For every $t \in [0,1]$, we write $u(t) = u_h(t) + \lambda(t) Z$. From \cref{lemma integral}, we get 
 \begin{equation*}
 q = \gamma_{u_h}(1) + \int^1_0 \lambda(t) \, \de t \cdot Z.
 \end{equation*}
 Moreover, since $\gamma_u$ is length-minimizing, we get
 \begin{equation} \label{lengu}
 \ell(\gamma_u) \stackrel{\eqref{lone_norm}}{=} \int^1_0 \sqrt{\rho(u_h(t),u_h(t))} \, \de t + \int^1_0 \lvert \lambda(t) \rvert \, \de t = d_F(\id,q) \ge 3c^{-1} + d(\id,Z)^2.
 \end{equation}
 First, we claim that \begin{equation} \label{first_claim}
 \int^1_0 \lvert \lambda(t) \rvert \, \de t \le d(\id,Z)^2. 
 \end{equation}
 Indeed, assume that $d(\id,Z)^2 < \int^1_0 \lvert \lambda(t) \rvert \, \de t $. Then, we estimate
 \begin{equation*}
 d\left(\id, \left(\int^1_0 \lambda(t) \, \de t \right) Z\right) = \sqrt{\left\lvert \int^1_0 \lambda(t) \, \de t \right\rvert} \cdot d(\id,Z) < \int^1_0 \lvert \lambda(t) \rvert \, \de t,
 \end{equation*}
 where in the equality we used that the metric space $(\H,d)$ is self-similar. The latter estimate implies that there exists a $\Delta$-horizontal curve $\gamma_{v} \colon [0,1] \to \H$, from $\id$ to $\left(\int^1_0 \lambda(t) \, \de t \right) Z$ such that $\ell(\gamma_{v}) < \int^1_0 \lvert \lambda(t) \rvert \, \de t$. Then, we consider the control $\widetilde{u} \colon [0,2] \to \mathfrak{h}$ defined by 
 \begin{equation*}
 \widetilde{u}(t) = \begin{cases}
 u(t) & \text{if $0\le t \le 1$,} \\ v(t-1) & \text{if $1<t\le2$.}
 \end{cases}
 \end{equation*}
 From \cref{lemma integral}, we get $\gamma_{\tilde{u}}(2) = \gamma_{u_h}(1) + \left(\int^1_0 \lambda(t) \, \de t \right) Z =q$. Moreover, we obtain
 $$\ell(\gamma_{\tilde{u}}) = \ell(\gamma_{u_h}) + \ell(\gamma_v) < \int^1_0 \sqrt{\rho(u_h(t),u_h(t))} \, \de t + \int^1_0 \lvert \lambda(t) \rvert \, \de t \stackrel{\eqref{lengu}}{=} \ell(\gamma_u),$$
 contradicting the minimality of $\gamma_u$. We thus proved \eqref{first_claim}. Consequently, thanks to \eqref{lengu}, we infer that $\ell(\gamma_{u_h}) \ge 3c^{-1}$.

 Next, we claim that $\gamma_{u_h}$ is a $d$-minimizing curve. Indeed, assume that there exists a $\Delta$-horizontal curve $\gamma_{u_h'}$ such that $\gamma_{u_h'}(1) = \gamma_{u_h}(1)$ and $\ell(\gamma_{u_h'}) < \ell(\gamma_{u_h})$. Then, the control $u' \colon [0,1] \to \mathfrak{h}$ defined by $u'(t) := u_h'(t) + \lambda(t)Z$ would satisfy $\gamma_{u'}(1)=q$ and $\ell(\gamma_{u'})<\ell(\gamma_u)$, contradicting the minimality of $\gamma_u$. We thus proved that $\gamma_{u_h}$ is $d$-minimizing and in particular $d(e,\gamma_{u_h}(1))= \ell(\gamma_{u_h}) \ge 3c^{-1}$.

 From \eqref{existence_time}, there exists $0 \le t_0 \le 1$ such that
 \begin{equation} \label{grater3}
 \lVert \pi(\gamma_h(t_0)) \rVert = c \cdot d(\id,\gamma_h(1)) \ge 3.
 \end{equation}
 Let $w \in \spn\set{X,Y}$ be such that $\rho(w,\pi(\gamma_h(t_0)))=0$ and $[w,\pi(\gamma_h(t_0))] = \left(\int^1_0 \lambda(t) \, \de t\right)Z$. We thus obtain 
 \begin{equation} \label{product_norm}
 \lVert w \rVert \cdot \lVert \pi(\gamma_h(t_0)) \rVert = \left\lvert \int^1_0 \lambda(t) \, \de t \right\rvert.
 \end{equation}
 We then consider the control $\widehat{u} \colon [0,3] \to \mathfrak{h}$ defined by
 \begin{equation*}
 \widehat{u}(t) = \begin{cases}
 w & \text{if $0 \le t \le 1$,} \\ u(t-1) & \text{if $1 < t \le 1+t_0 $,} \\ -w & \text{if $1+t_0 < t \le 2+t_0$,} \\ u(t-2) & \text{if $2+t_0 < t \le 3$.}
 \end{cases}
 \end{equation*}
Note that the curves $\gamma_{\hat{u}}$ and $\gamma_u$ have the same endpoints. We then estimate
 \begin{eqnarray*}
 \ell(\gamma_{\hat{u}}) &=& \ell(\gamma_{u_h}) + 2 \lVert w \rVert \\
 &\stackrel{\eqref{product_norm}}{=}& \ell(\gamma_{u_h}) + 2\frac{\left\lvert \int^1_0 \lambda(t) \, \de t \right\rvert}{\lVert \pi(\gamma_h(t_0)) \rVert} \\
 &\stackrel{\eqref{grater3}}{\le}& \ell(\gamma_{u_h}) + \frac{2}{3} \int^1_0 \lvert \lambda(t) \rvert \, \de t \\
 &\le& \ell(\gamma_{u_h}) + \int^1_0 \lvert \lambda(t) \rvert \, \de t = \ell(\gamma_u).
 \end{eqnarray*}
 From the minimality of $\gamma_u$, we get that $\ell(\gamma_{\hat{u}})=\ell(\gamma_u)$ and all inequalities in the previous estimates are equalities. In particular $\int^1_0 \lvert \lambda(t) \rvert \, \de t = 0$ and therefore $\gamma_u$ is a $\Delta$-horizontal curve. We conclude that $d(\id,q)=d_F(\id,q)$ and the statement is proved by setting $r_0 \coloneq 3c^{-1} + d(\id,Z)^2$.
\end{proof}

\subsection{Precise volume growth of the Riemannian Heisenberg group}
If $(G,d)$ is a 2-step nilpotent sub-Riemannian Lie group that is not a Carnot group, then \cref{asymptotic_sphere} implies that
\begin{equation} \label{double}
 Cr^Q + D_1r^{Q-2} + o(r^{Q-2}) \le \vol(B(r)) \le Cr^Q + D_2r^{Q-2} + o(r^{Q-2}), \quad \text{as $r \to \infty$,}
\end{equation}
for some $C,D_1,D_2,Q > 0$. We show that, for the Riemannian Heisenberg group, one can actually take $D_1=D_2$. In the following example, we explicitly compute that in this case one has 
 \begin{equation} \label{volume_heis}
 \vol(B(r))=C_4 r^4 + C_2 r^2 + C_0, \quad \text{for every $r > 2 \pi$,}
 \end{equation}
 for some constants $C_0,C_2,C_4 \neq 0$. In general, it is unclear if one should expect 
 $$\vol(B(r)) = Cr^Q + Dr^{Q-2} + o(r^{Q-2}), \quad \text{as $r \to \infty$.}$$
 to hold for every 2-step nilpotent sub-Riemannian Lie group, nor what can be said about the error term.
 \begin{example}
 We consider the left-invariant Riemannian metric $d_R$ on the Heisenberg group $\H$ such that $\set{X,Y,Z}$ is an orthonormal frame for the metric tensor $\rho$. We fix $r > 0$ and we study the volume of the metric ball
 \begin{equation*}
 B(r) \coloneq \set{ p \in \H \, \colon \, d_R(e,p) \le r}.
 \end{equation*}
 
 We identify $\H$ with $\R^3$ with the global diffeomorphism
 \begin{equation*}
 \Phi \colon \R^3 \to \H, \quad (x,y,z) \mapsto \exp(xX + yY + zZ).
 \end{equation*}
 We remark that the pushforward by $\Phi$ of the Lebesgue measure on $\R^3$ is a Haar measure on $\H$.
 
 With this identification, the function $f \coloneq d(e,\cdot)$ has a rotational symmetry around the $Z$ axis and a reflectional symmetry along the $XY$ plane. Therefore $B(r)$ is completely determined by 
 \begin{equation*}
 E(r) \coloneq \set{(x,0,z) \in \R^3 \, \colon x,z>0, \, d_R(e,(x,0,z))=r}.
 \end{equation*}
 As a consequence of the explicit integration of the geodesic equations \cite[Proposition 5.4]{MR3673666}, every point $p \in E(r)$ is reached by a unique length-minimizing curve from the identity element $e=(0,0,0)$, and such curve is of the form $\gamma_p =(x_p,y_p,z_p) \colon [0,t_p] \to \R^3$, where
 \begin{equation*}
 \begin{cases}
 x_p(t)=\frac{1}{k_p}\left((\cos \theta_p(\cos(k_pt)-1) - \sin \theta_p \sin(k_pt)\right) ,\\
 y_p(t)=\frac{1}{k_p}\left((\sin \theta_p(\cos(k_pt)-1) + \cos \theta_p \sin(k_pt)\right) ,\\
 z_p(t)= \frac{t}{2k_p}-\frac{1}{2k_p^2}\sin(k_pt)+k_pt \,,
 \end{cases}
 \end{equation*}
 for some $k_p>0$, $0 \le \theta_p < \pi$, and $0 < t_p < \frac{2\pi}{k_p}$. Therefore, we get
 \begin{equation*}
 p = \left(\frac{\sqrt{2-2\cos(k_pt_p)}}{k_p}, 0, \frac{k_pt_p - \sin(k_pt_p)}{2k_p^2}+k_pt_p\right).
 \end{equation*}
We observe that $\gamma_p$ is parametrized with constant speed $\sqrt{1+k_p^2}$, thus we require $t_p\sqrt{1+k_p^2}=r$.

 We set $\alpha_p \coloneq k_p t_p$. The condition $0 < t_p < \frac{2\pi}{k_p}$ implies $\alpha_p < 2 \pi$, whilst the condition $t_p\sqrt{1+k_p^2}=r$ implies $\alpha_p < r$. Vice-versa, we claim that for every $0 < \alpha < \min\set{2 \pi, r}$, there exists a unique $k>0$ and $0 < t < \frac{2\pi}{k}$ such that $kt=\alpha$ and $t\sqrt{1+k^2}=r$. Indeed, after some easy computations, these two conditions are equivalent to
 \begin{equation*}
 \frac{k}{\sqrt{1+k^2}}=\frac{\alpha}{r}, \quad t=\sqrt{r^2-\alpha^2},
 \end{equation*}
 which admit unique solutions $k,t>0$, since $\alpha < r$. Moreover, since $\alpha < 2 \pi$, we get $t < \frac{2 \pi}{k}$.

 We conclude that the curve $p \coloneq (0,\min\set{2 \pi,r}) \to \R^3$, defined by setting
 \begin{equation*}
 p(\alpha) \coloneq \left(\frac{\sqrt{r^2-\alpha^2}\sqrt{2-2\cos(\alpha)}}{\alpha},0,\frac{(r^2-\alpha^2)(\alpha-\sin(\alpha))}{2\alpha^2}+\alpha \right),
 \end{equation*}
 parametrizes the set $E(r)$.

 The volume $V(r)$ of the ball $B(r)$ is therefore twice the volume of the solid of revolution with profile $p(\alpha)$ around the $Z$ axis. Thus
 \begin{equation*}
 V(r) = 2 \pi \int_0^{\min\set{2\pi,r}} x(\alpha)^2 \frac{\partial z(\alpha)}{\partial \alpha} \, \de \alpha .
 \end{equation*}
 After some straightforward computations, we finally get
 \begin{align*}
 V(r) &= 2 \pi r^4 \int_0^{\min\set{2\pi,r}} \frac{1-\cos(\alpha)}{\alpha^2}\cdot \frac{2\sin(\alpha)-\alpha\cos(\alpha)-\alpha}{\alpha^3} \, \de \alpha \\ 
 & + 2 \pi r^2 \int_0^{\min\set{2\pi,r}} \frac{\sin^2(\alpha)}{\alpha^2} - \frac{(1 - \cos(\alpha))(2\sin(\alpha)-\alpha\cos(\alpha)-\alpha)}{\alpha^3} \, \de \alpha \\
 & - 2 \pi \int_0^{\min\set{2\pi,r}} \sin^2(\alpha) \, \de \alpha .
 \end{align*}
 We stress that, as expected $V(r)=O(r^3)$, as $r \to 0$. Whilst for $r \ge 2 \pi$, we get \eqref{volume_heis}
 for some $C_0,C_2,C_4 \neq 0$.
 \end{example}



\subsection{A vast class of rough isometries} \label{vast class}
\cref{equivalence quasi_isometries} allows us to determine a class of rough isometries made of automorphisms. Given a simply connected, $2$-step nilpotent sub-Riemannian Lie group $(G,d)$ and a map $\vphi \in \mathrm{Aut}(G)$, we define a new metric $d'$ as
 \begin{equation*}
 d'(p,q) := d(\vphi(p),\vphi(q)), \quad \text{for every $p,q \in G$.}
 \end{equation*}
 The metric space $(G, d')$ is again a sub-Riemannian Lie group, with sub-Riemannian structure given by the pull-back of the sub-Riemannian structure of $(G,d)$. Let us denote by $\lVert \cdot \rVert_\mathrm{ab}$ and $\lVert \cdot \rVert'_\mathrm{ab}$ the abelianization norms of $(G,d)$ and $(G,d')$, respectively.

 If we further assume that $\vphi_\mathrm{ab}$ is an isometry for $\lVert \cdot \rVert_\mathrm{ab}$, then $\lVert \cdot \rVert_\mathrm{ab} = \lVert \cdot \rVert'_\mathrm{ab}$ and \cref{equivalence quasi_isometries} implies that $\vphi$ is a rough isometry. We summarize our discussion in the following result.

 \begin{proposition} \label{many quasi isometries}
 Let $(G,d)$ be a simply connected, $2$-step nilpotent sub-Riemannian Lie group. If $\vphi \in \mathrm{Aut}(G)$ is such that $\vphi_\mathrm{ab}$ is an isometry for the abelianization norm of $(G,d)$, then $\vphi$ is a rough isometry. 
\end{proposition}

Thanks to \cref{many quasi isometries}, we identify a class of rough isometries $\vphi \in \mathrm{Aut}(G)$ of a 2-step nilpotent, sub-Riemannian Lie group $(G,d)$ that are not at bounded distance from any isometry. To determine such maps, we make use of the following general result.

\begin{proposition} \label{bounded iso}
 Let $(G,d)$ be a simply connected, nilpotent, metric Lie group for which all bounded sets are precompact. Let $\vphi$ be an automorphism of $G$ such that $\vphi_\mathrm{ab}=\mathrm{id}$. The following are equivalent:
 \begin{enumerate}[(i)]
 \item $\vphi$ is at bounded distance from an isometry of $(G,d)$,
 \item $\vphi=C_g$ for some $g \in G$, where $C_g$ is defined by $C_g \colon h \mapsto ghg^{-1}$.
 \end{enumerate}
\end{proposition}

Our proof of \cref{bounded iso} relies on the following observation.

\begin{lemma} \label{bounded auto}
 Let $(G,d)$ be a simply connected, nilpotent, metric Lie group for which all bounded sets are precompact. If $\vphi, \psi \in \mathrm{Aut}(G)$ are at bounded distance, then $\vphi = \psi$.
\end{lemma}

\begin{proof}
 Assume $\vphi(p) \neq \psi(p)$ for some $p \in G$. 
 Consider the function $n\in \N \mapsto \vphi(p)^{-n}\psi(p)^n$, in exponential coordinates. By Baker-Campbell-Hausdorff formula (see \cite[Theorem~11.2.6]{hilgert2011structure}), 
 this function is a polynomial. Moreover, it is not constant since it has different values at $n=0$ and $n=1$.
 Consequently, it diverges as $n$ goes to infinity.
 By the assumption on $(G,d)$ the set is unbounded with respect to the distance $d$, therefore 
 \begin{equation*}
 d(\vphi(p^n),\psi(p^n)) = d( \id , \vphi(p)^{-n}\psi(p)^n) \to \infty, \quad \text{ as $n \to \infty$.}
 \end{equation*}
 Thus, we proved that the automorphisms $\vphi$ and $\psi$ are not at bounded distance.
\end{proof}

\begin{proof}[Proof of \cref{bounded iso}]
 First, assume that $\vphi=C_g$, for some $g \in G$. Then $\vphi$ is the composition of the map $L_g$, which is an isometry, and the right-translation $R_g$, which is at bounded distance from the identity map. We conclude that $\vphi$ is at bounded distance from the isometry $L_g$.

 Conversely, assume that $\vphi$ is at bounded distance from an isometry $f \colon (G,d) \to (G,d)$, i.e., there exists $C>0$ such that
 \begin{equation} \label{bounded 1}
 d(f(p),\vphi(p)) \le C, \quad \text{for every $p \in G$.}
 \end{equation}

Since $f$ is an isometry of a connected metric nilpotent Lie group, then $f$ is affine, i.e., it is of the form $f = L_g \circ \widetilde{f}$ for some $g \in G$ and some isometry $\widetilde{f} \in \mathrm{Aut}(G)$; see \cite[Theorem 1.2]{MR3646026}.

 We claim that the abelianization map $\widetilde{f}_\mathrm{ab}$ of $\widetilde{f}$ is the identity map. Indeed, since the map $\pi_\mathrm{ab}$ is a submetry for the quotient metric $d_\mathrm{ab}$, we get
 \begin{eqnarray*}
 C &\stackrel{\eqref{bounded 1}}{\ge}& d(f(p),\vphi(p)) \ge d_\mathrm{ab}(\pi_\mathrm{ab} \circ f (p),\pi_\mathrm{ab} \circ \vphi (p)) \\ &=& d_\mathrm{ab}(L_{\pi_\mathrm{ab}(g)} \circ \widetilde{f}_\mathrm{ab}(\pi_\mathrm{ab} (p)), \pi_\mathrm{ab} (p) ), \quad \text{for every $p \in G$,}
 \end{eqnarray*}
 where we used that $\vphi_\mathrm{ab}=\mathrm{id}$. Therefore,
 \begin{equation*}
 d_\mathrm{ab}(\widetilde{f}_\mathrm{ab} \circ \pi_\mathrm{ab} (p), \pi_\mathrm{ab} (p)) \le C + d_\mathrm{ab}(\id,\pi_\mathrm{ab}(g)), \quad \text{for every $p \in G$.}
 \end{equation*}
 The latter estimate implies that $\widetilde{f}_\mathrm{ab}$, which is a linear map on the abelian group $G_\mathrm{ab}$, is at bounded distance from the identity map. We conclude that $\widetilde{f}_\mathrm{ab}$ is the identity map. 
 Thanks to \cite[Lemma~2.5]{MR3646026}, the map $\widetilde{f}$, which is an isometry of $(G,d)$, is also an isometry of $(G,d_R)$, for some Riemannian metric $d_R$ on $G$. Since $\de_\id \widetilde{f}$ is a Lie algebra automorphism of $\g$, then $\de_\id \widetilde{f}([\g,\g])=[\g,\g]$. Since $\de_\id \widetilde{f}$ is an isometry for a Euclidean norm, we also get $\de_\id \widetilde{f}([\g,\g]^\perp)=[\g,\g]^\perp$. Thanks to \cref{equality automorphism}, we infer that $\widetilde{f}$ is the identity map as well, thus $f = L_g$ for some $g \in G$.

 In particular, since right-translations are at bounded distance from the identity map, we obtain that $\vphi$ is at bounded distance from $C_g \in \mathrm{Aut}(G)$. Thanks to \cref{bounded auto}, we conclude that $\vphi =C_g$.
 \end{proof}

\begin{remark}
 We point out that, in every simply connected, $2$-step nilpotent Lie group $G$, the map
\begin{equation} \label{vast class 2}
 \vphi \colon G \to G, \quad g \mapsto g \cdot L(\pi_\mathrm{ab}(g))
\end{equation}
is a Lie group automorphism of $G$, for every linear map $L \colon G_\mathrm{ab} \to [G,G]$. In particular, an automorphism $\vphi$ as in \eqref{vast class 2} satisfies $\vphi_\mathrm{ab}=\mathrm{id}$. We conclude that the class of Lie group automorphisms $\vphi$ satisfying the hypothesis of \cref{many quasi isometries} contains an abelian subgroup isomorphic to the vector space of linear maps from $G_\mathrm{ab}$ to $[G, G]$. 

A dimensional argument shows that, as soon as $\dim([G ,G])>1$ or $[G, G]$ is strictly contained in the center $Z(G)$ of $G$, then there are maps $\vphi \in \mathrm{Aut}(G)$ as in \eqref{vast class 2} that are not of the form $C_g$ for any $g \in G$. Combining \cref{many quasi isometries} with \cref{bounded iso}, we infer that those maps $\vphi$ are rough isometries that are not at bounded distance from any isometry.
\end{remark}

We point out that the class of simply connected Lie groups $G$ such that $\dim([G, G])=1$ and $Z(G)=[G, G]$ is precisely the family of Heisenberg groups $\H_n$, for $n \ge 1$.
As a consequence of the above discussion, we have the following result.

\begin{theorem}
 Let $G$ be a simply connected, $2$-step nilpotent Lie group that is not isomorphic to any Heisenberg group $\H_n$, for $n \ge 1$.
 Then, there is an automorphism $\vphi \in \mathrm{Aut}(G)$ that is a rough isometry and is not at bounded distance from any isometry.
\end{theorem}

The above result and the next examples are highly in contrast with the fact that in Euclidean spaces, every rough isometry (not necessarily an automorphism) is at bounded distance from an isometry; see \cite[Theorem~3]{MR511409}.

\begin{example} \label{final example}
 We exhibit an explicit example of a Lie group automorphism of a sub-Riemannian Lie group that is a rough isometry, but it is not at bounded distance from any isometry. We consider $G = \H \times \R$ to be the simply connected $2$-step nilpotent Lie group whose associated Lie algebra $\g$ admits a basis $\set{X,Y,Z,T}$ such that the only non-trivial bracket relation between elements of the basis is $[X,Y]=Z$.
 We consider on $G$ the sub-Riemannian structure $(\Delta,\rho)$ such that $\set{X,Y,T}$ is an orthonormal frame for $\rho$, and we denote by $d$ the corresponding sub-Riemannian metric. We identify $\H \times \R$ with $\R^4$ with the global diffeomorphism
 \begin{equation*}
 \Phi \colon \R^4 \to \H \times \R, \quad (x,y,z,t) \mapsto \exp(xX + yY + zZ + tT).
 \end{equation*}
Let $\vphi \in \mathrm{Aut}(G)$ be the Lie group automorphism such that, in coordinates,
 \begin{equation*} 
 \vphi(x,y,z,t)=(x,y,z+t,t).
 \end{equation*}
 On the one hand, we observe that $\vphi_\mathrm{ab}$ is the identity map and therefore, from \cref{many quasi isometries}, we get that $\vphi$ is a rough isometry. On the other hand $p = (0,0,0,1)$ is in the center of $G$, but $\vphi(p) \neq p$. We infer that $\vphi$ is not of the form $C_g$ for any $g \in G$. In view of \cref{bounded iso}, we conclude that $\vphi$ is not at bounded distance from an isometry.
\end{example}
\begin{example} \label{Giorgi example}
In the literature, there have been other constructions of rough isometries that are not at bounded distance from isometries. For example, it is easy to show that in the Heisenberg group equipped with any left-invariant geodesic distance, an example is
$$(x,y,z)\mapsto (x,y,z+|x|).$$
An observation is that on each of the two subsets $\{x\leq 0\}$ and $\{x\geq 0\}$, this map coincides with some conjugation.
Other examples in the Heisenberg group have been studied in the EPFL master thesis of Alessio Giorgi. We have that, for every Lipschitz function $\psi:\R^2\to \R$, the map $$(x,y,z)\mapsto (x,y,z+\psi(x,y))$$ is a
rough isometry and it is not at bounded distance from any isometry.
A classification of the rough isometries of the Heisenberg group is missing.
\end{example}

\bibliographystyle{abbrv}	
\bibliography{biblio}

\end{document}